\documentclass[11pt]{article}

\usepackage{graphicx}
\usepackage{amssymb,epsfig}
\usepackage{amsmath}
\usepackage[francais,english]{babel}
\usepackage[latin1]{inputenc}
\usepackage[T1]{fontenc}

\newcommand{\dps}[1]  {\displaystyle{#1} }

\def\div{{\rm div \;}}

\def\R{\mathbb{R}}
\def\T{\mathbb{T}}
\def\Z{\mathbb{Z}}

\def\E{\mathbb{E}} %esp\'erance
\def\I{\mbox{Id}} % l'identité
\newtheorem{theorem}{Theorem}
\newtheorem{proposition}{Proposition}
\newtheorem{corollary}{Corollary}
\newtheorem{definition}{Definition}
\newtheorem{lemma}{Lemma}
\newtheorem{remark}{Remark}

\def\acknowledgement{\medskip {\bf {Acknowledgements} : }}

\title{Long-time convergence of an Adaptive Biasing Force method}
\author{Tony Lelièvre$^{(1,2)}$, Felix Otto$^{(3)}$, Mathias Rousset$^{(1,2)}$ and Gabriel Stoltz$^{(1,2)}$\\
\footnotesize{(1) CERMICS, Ecole Nationale des
  Ponts (ParisTech), 6 \& 8 Av. B. Pascal,
  77455 Marne-la-Vall\'ee, France.}\\
\footnotesize{(2) INRIA Rocquencourt, MICMAC project-team, B.P. 105, 78153 Le Chesnay Cedex, France.}\\
\footnotesize{(3) Institute for Applied Mathematics,
 University of Bonn, Wegelerstrasse 10, 53115 Bonn,
 Germany.} \\
\{lelievre,rousset,stoltz\}@cermics.enpc.fr ~~~ otto@iam.uni-bonn.de}

\begin{document}

%\selectlanguage{francais}
\selectlanguage{english}

\maketitle

\abstract{We propose a proof of convergence of an adaptive method used in molecular dynamics to compute free energy profiles
  (see~\cite{darve-pohorille-01,henin-chipot-04,lelievre-rousset-stoltz-07-b}). Mathematically, it
  amounts to studying the long-time behavior of a stochastic process
  which satisfies a non-linear stochastic differential equation, where
  the drift depends on conditional expectations of some functionals of the process. We use
  entropy techniques to prove exponential convergence to the
  stationary state.}

% non-linear Fokker-Planck equation, long-time behaviour, entropy
% techniques, Adaptive Biasing Force, metadynamics

% 35B40, 37M25, 60K35

\medskip

%--------------------------------
%
%        INTRODUCTION
% 
%---------------------------------

\section{Introduction}

In Section~\ref{sec:FED}, we introduce the physical context of this
work, namely molecular dynamics and the computation of free energy
differences in the canonical statistical ensemble. In Section~\ref{sec:ABF},
we introduce the adaptive dynamics we study and the main results we
prove are presented in Section~\ref{sec:pres_res}.

\subsection{Computations of free energy differences and metastability}\label{sec:FED}

Let us  consider the Gibbs-Boltzmann  measure 
\begin{equation}\label{eq:mu}
d\mu(q)=Z^{-1}\exp(-\beta V(q))  dq,
\end{equation}
where $q  \in  {\mathcal D}$,  $V:{\mathcal D}  \to \R$,  $Z=\int_{{\mathcal D}}
\exp(-\beta  V(q))  \,  dq$ and ${\mathcal D}=\{ q,\, V(q) < \infty \}$
is the configuration space. In  the  applications  we  consider,  $q$
represents the  position of~$N$ particles so that, in the following, ${\mathcal D}$ is
 an open subset (possibly the whole) of $\R^{n}$, with $n=3N$. All the results we prove are
 also satisfied if ${\mathcal D}$ is an open subset of $\T^{n}$ (where
 $\T=\R /  \Z$ denotes the  one-dimensional torus). The function~$V$ is the
energy  associated with  the positions  of the  particles and  $\beta$ is
proportional to the inverse of the temperature.  The probability measure
$\mu$ represents the equilibrium measure sampled by the particles in the
canonical  statistical ensemble.   A typical  dynamics  that can  be used  to
sample this measure is
\begin{equation}
\label{eq:Q}
dQ_t=- \nabla V(Q_t) \, dt + \sqrt{2 \beta^{-1}} dB_t,
\end{equation}
where $B_t$ is a $n$-dimensional standard Brownian motion. More
generally, for any smooth positive function $\gamma:{\mathcal D}  \to
\R_+^*$, the stochastic process $Q_t$ which satisfies
\begin{equation}
\label{eq:Q_prime}
dQ_t=- \nabla (V - \beta^{-1}\ln \gamma)(Q_t) \gamma(Q_t) \, dt
+ \sqrt{2 \beta^{-1} \gamma(Q_t)} dB_t
\end{equation}
samples the measure $\mu$.

Let  us introduce  a so-called {\em  reaction coordinate}  $\xi:{\mathcal D}  \to
{\mathcal M}$, with ${\mathcal M}=\R$ or ${\mathcal M}=\T$.   For a  given
configuration  $q$, $\xi(q)$  represents  a coarse-grained  information,
which is valuable from a  physical point of view. For instance, $\xi(q)$
may be a dihedral angle, for example to characterize the conformation of a molecule, in
which case ${\mathcal M}=\T$, or
the signed distance to an hypersurface of ${\mathcal D}$ (characterizing
a transition state), for example to  measure  the  evolution  of a chemical reaction, in
which case ${\mathcal M}=\R$. The function $\xi$ is therefore related  to some macroscopic
information of  the system.  Usually,  in~(\ref{eq:Q}), the time-scale
for the  dynamics on  $\xi(Q_t)$ is larger  than the time-scale  for the
dynamics on  $Q_t$ (due  to metastable states),  so that $\xi$  can also
be understood  as a function such  that $\xi(Q_t)$ is  a slow variable
compared to $Q_t$. 

In the following, we suppose that
\begin{equation*}
  \label{eq:hyp_xi}
  \text{{\bf [H1]}~~~$\xi$ is a smooth function such that $|\nabla \xi|>0$ on ${\mathcal D}$.}
\end{equation*}
Thus, the subsets $\Sigma_z=\{x \in {\mathcal D},\,  \xi(x)=z\}$ of
${\mathcal D}$ are smooth submanifolds of  co-dimension   one which
define a partition of ${\mathcal D}$:
$${\mathcal D}=\bigcup_{z \in {\mathcal M}} \Sigma_z \mbox{ and }
\Sigma_z \cap \Sigma_{z'}= \emptyset \mbox{ for $z \neq z'$}.$$  We denote  by
$\sigma_{\Sigma_z}$ the surface measure on $\Sigma_z$,
\emph{i.e.}  the Lebesgue measure on~$\Sigma_z$ induced by the
Lebesgue measure in the ambient space $\mathcal{D} \supset
\Sigma_z$. The submanifold~$\Sigma_z$ naturally has a (complete and locally compact) Riemannian structure induced by the Euclidean
structure of the ambient space~${\mathcal D}$.

The  image of the  measure $\mu$  by $\xi$  is
$\frac{\exp(-\beta  A(z))\, dz}{\int_{\mathcal M} \exp(-\beta  A(z))\, dz}$
where $A$ is the so-called {\em free energy} defined by:
\begin{equation}
\label{eq:A}
A(z)=-\beta^{-1} \ln (Z_{\Sigma_z})
\end{equation}
where
$$Z_{\Sigma_z}=\int_{\Sigma_z}  |\nabla \xi|^{-1} \exp(-\beta V) d
\sigma_{\Sigma_z}.$$
We assume henceforth that $\xi$ and $V$ are such that $Z_{\Sigma_z}<
\infty$. The free energy is actually defined up to an additive constant,
the quantity  $\exp(-\beta  A)$ being then defined up to a
multiplicative constant, which disappears in the normalization of the
probability measure $\frac{\exp(-\beta  A(z))\, dz}{\int_{\mathcal M} \exp(-\beta  A(z))\, dz}$.
Many algorithms in molecular dynamics~\cite{chipot-pohorille-07} aim to compute the
image of the  measure $\mu$  by $\xi$, which amounts to compute free
energy differences, namely quantities of the form $A(z)-A(z_0)$.  This
is typically obtained by computing (and then integrating) the derivative $A'(z)$, 
called the {\em mean force}. Using the co-area formula (see
Appendix~\ref{sec:co-area}), the following expression for $A'(z)$ can be
obtained (see~\cite{ciccotti-lelievre-vanden-einjden-06}, or the proof
of Lemma~\ref{lem:psi_bar_prime} below):
\begin{equation}
\label{eq:A_prime}
\boxed{
A'(z)=Z_{\Sigma_z}^{-1} \int_{\Sigma_z} F \, |\nabla
\xi|^{-1}  \exp(-\beta V) d \sigma_{\Sigma_z},
}
\end{equation}
where $F$ is the so-called {\em local mean force} defined by
\begin{equation}
\label{eq:F}
\boxed{
F=\left( \frac{\nabla V \cdot \nabla \xi}{|\nabla \xi|^{2}}
  - \beta^{-1} \div\left(\frac{\nabla \xi}{|\nabla \xi|^{2}}\right) \right).
}
\end{equation}

This can be rewritten in terms of conditional expectation as: For a
random variable~$X$ with law $\mu$,
\begin{equation}
\label{eq:A_prime_bis}
A'(z)=\E \left( F (X) \, \Big| \xi(X)=z \right).
\end{equation}

In practice, free energy profiles are used for example to
compare the likelihood of various conformations of a molecule, or to
compute the
rate of a chemical reaction. Free energy can also be useful to compute ensemble averages in
the canonical ensemble using the following formula (which is a conditioning formula): For any function $\phi :
{\mathcal D} \to \R$,
\begin{equation}\label{eq:moy_NVT}
\int_{\mathcal D} \phi \, d \mu =\frac{\displaystyle{\int_{\mathcal M} \int_{\Sigma_z} \phi \, d \mu_{\Sigma_z} \exp(-\beta A(z)) \, dz}}{\displaystyle{\int_{\mathcal M} \exp(-\beta A(z)) \, dz} },
\end{equation}
where $\mu_{\Sigma_z}$ is the probability measure $\mu$ conditioned to a
fixed value $z$ of the reaction coordinate:
\begin{equation}
d\mu_{\Sigma_z} = Z_{\Sigma_z}^{-1} |\nabla \xi|^{-1} \exp(-\beta V) d
  \sigma_{\Sigma_z}.
\end{equation}
Notice that~(\ref{eq:A_prime}) also writes $A'(z)=\int_{\Sigma_z} F \, d\mu_{\Sigma_z}$. Equation~(\ref{eq:moy_NVT}) may be interesting to compute averages in the canonical ensemble
 since, if
the reaction coordinate is well chosen, it is expected that the sampling
of the conditioned probability measure $\mu_{\Sigma_z}$ is easier than
the sampling of $\mu$ (the metastable features of the measure $\mu$
being mostly in the direction of the reaction coordinate $\xi$). The
sampling of $\mu_{\Sigma_z}$ can be done for example by projection of the gradient
dynamics on $\Sigma_z$
(see~\cite{ciccotti-lelievre-vanden-einjden-06}). The quantity $\int_{\Sigma_z}
\phi \, d \mu_{\Sigma_z}$ can thus be evaluated by an efficient Monte Carlo
procedure, and the computation of $\int_{\mathcal D} \phi \, d \mu$
through~(\ref{eq:moy_NVT}) then only requires a
one-dimensional integration, and the computation of the free energy (up
to an additive constant).

Due to the  high dimensionality of the problem  (the number of particles
$N$  is usually  very large),  methods to  compute mean  forces  or free
energy differences are  Monte Carlo methods. They typically  rely on the
simulation   of  a  diffusion Markov   process.  The   most  recent   methods  use
non-homogeneous or  non-linear Markov processes.  Classical examples are
exponential  reweighting of  non-equilibrium paths  (based upon the
so-called Jarzynski equality,
see~\cite{jarzynski-97,lelievre-rousset-stoltz-07-a})      or     adaptive
methods
(see~\cite{darve-pohorille-01,henin-chipot-04,iannuzzi-laio-parrinello-03,wang-landau-01}). 

We are interested here in adaptive methods to  compute free
energy differences, and more precisely
Adaptive Biasing
Force  techniques  (see~\cite{darve-pohorille-01,henin-chipot-04}).  The
principle of adaptive methods is  to modify the potential $V$ during the
simulation, in  order to  remove the metastable  features of  the simple
dynamics~(\ref{eq:Q}),  while approximating  the free  energy~$A$. Many
methods       have      been       proposed      and       we      refer
to~\cite{lelievre-rousset-stoltz-07-b}  for a  unified presentation  of these
techniques, as well as a discussion of efficient parallel implementations.
The aim of this paper is to propose a mathematical study of the
Adaptive Biasing Force method to give a rigorous formulation and proofs of the
following statements (which are the main arguments of practitioners
of the field to advocate the use of adaptive methods):
\begin{itemize}
\item[{\bf [S1]}] The adaptive biasing force technique helps to remove the
  metastable features of the  simple dynamics~(\ref{eq:Q}), and thus
  enables efficient exploration of the configuration space.
\item[{\bf [S2]}] With the adaptive biasing force technique, the free  energy  $A$
  is obtained in the longtime limit, and the convergence is
  exponentially fast in time.
\end{itemize}

\subsection{An Adaptive Biasing Force technique}\label{sec:ABF}

The     Adaptive     Biasing      Force (ABF)     method     was     introduced
in~\cite{darve-pohorille-01,henin-chipot-04} and is recast in a general
mathematical framework in~\cite{lelievre-rousset-stoltz-07-b}.  We propose to  study here
one version of this method, applied to the context of Brownian (or
overdamped Langevin)
dynamics\footnote{Such methods can also be applied for other dynamics,
  like Langevin dynamics. We only consider
  Brownian dynamics in this paper.}.

The ABF dynamics we propose to study is the
following non-linear stochastic differential equation:
\begin{equation}\label{eq:X}
\boxed{
\begin{array}{l}
dX_t= -  \nabla \Big(V - A_t \circ \xi +W \circ \xi - \beta^{-1} \ln (|\nabla \xi|^{-2}) 
 \Big)(X_t) \, |\nabla \xi|^{-2}(X_t)
\,dt  \\[5pt]
\phantom{dX_t=} + \sqrt{2 \beta^{-1}}  |\nabla \xi|^{-1} (X_t) dB_t,
\end{array}
}
\end{equation}
where $W$ is an additional well-chosen potential that we will define below and $A_t$ is the
``free energy observed at time $t$''. More precisely, the
derivative of $A_t$ with respect to the reaction coordinate is defined
as (compare with~(\ref{eq:A_prime_bis})):
$\forall z \in {\mathcal M}$,
\begin{equation}\label{eq:A_prime_t}
\boxed{
  A_t'(z)=\E\left(F(X_t) \, \Big| \xi(X_t)=z \right),
}
\end{equation}
where $F$ is defined by~(\ref{eq:F}). With a slight abuse of
terminology, the function $A_t'$ is called the
{\em biasing force}.
Notice that here and in the following, the notation $'$ denotes a derivative with
respect to the reaction coordinate values, while the notation $\circ$
denotes the composition operator. Equation~(\ref{eq:A_prime_t})
defines $A_t$ up to an additive (time-dependent) constant, which does not modify~(\ref{eq:X}). 

Compared to the simple dynamics~(\ref{eq:Q}), three modifications have
been made to obtain~\eqref{eq:X}--\eqref{eq:A_prime_t}:
\begin{enumerate}
\item First and foremost, the potential $V$ has been changed
to the biasing potential $V-A_t \circ \xi$. This is the bottom line of
the adaptive strategy. The algorithm we study here is prototypical of many adaptive methods used in molecular dynamics
(see~\cite{lelievre-rousset-stoltz-07-b}). In the original Adaptive Biasing Force technique as presented
in~\cite{darve-pohorille-01,henin-chipot-04}, the conditional
expectation~\eqref{eq:A_prime_t} is actually ``approximated'' by some
conditional averages over one single trajectory. The dynamics we study
here is not clearly related with such a discretization, but rather with a
discretization of~\eqref{eq:A_prime_t} using an interacting particle system, where
many replicas of the system contribute to the free energy profile
(see~\cite{lelievre-rousset-stoltz-07-b}).
\item Second, a potential $W \circ \xi$ has been added. This is actually
  needed only in the case when $\mathcal{M}$ is an unbounded domain  (we
  recall that $\mathcal{M}$ is the domain where the reaction coordinate
  lives). In theses cases, $W$ is chosen so that the law of $\xi(X_t)$
  converges exponentially fast to its longtime limit (more precisely, the Fisher information
  associated with this law converges exponentially fast to zero, see~[H4] below for a more detailed
  statement). Besides, from a numerical point of view, such a potential
  is sometimes used in practice in order to separately sample some parts of the reaction coordinate
  space ${\mathcal M}$ (as in stratified sampling strategies).
\item Third, some
terms depending on $|\nabla \xi|$ have been introduced. This modification is made in order to obtain a simple diffusive behavior for
the law of~$\xi(X_t)$ (see Proposition~\ref{prop:psi_xi} below).  It is expected that
the longtime convergence of $A_t'$ towards $A'$ still holds without this
modification, by simply considering the gradient dynamics 
\begin{equation}\label{eq:original}
dX_t=- \nabla
(V - A_t \circ \xi + W \circ \xi)(X_t) \, dt + \sqrt{2 \beta^{-1}}
dB_t,
\end{equation}
with the same definition~(\ref{eq:A_prime_t}) for $A_t'$.
However, we are only able to prove a weaker convergence result in
this case. This is the matter of Sections~\ref{sec:original} and~\ref{sec:proof_original}. Notice that
if $|\nabla \xi|$ is constant (for example if $\xi$ is a length), a
simple change of time relates~\eqref{eq:original}
with~\eqref{eq:X}. Notice also that if we take $A_t=W=0$
in~(\ref{eq:X}), then $X_t$ samples the original Gibbs measure
$\mu$ defined by~(\ref{eq:mu}) (see Equation~(\ref{eq:Q_prime}) above).
\end{enumerate}

\begin{remark}[On the computation of $A_t'(z)$]
From a practical point of view, with the additional terms
mentioned in item 3 above, it is
possible to compute the biasing force $A_t'(z)$ without explicitly
evaluating $F$ since (by It\^o's calculus on $X_t$ that satisfies~(\ref{eq:X}),
and assuming $W=0$ for simplicity)
\begin{equation}\label{eq:SDE_xi}
F(X_t) \, dt = d \xi (X_t) + A_t'(\xi(X_t))\, dt - \sqrt{2
\beta^{-1}}  \frac{\nabla \xi}{|\nabla \xi|} (X_t) \cdot dB_t.
\end{equation}
By a simple finite difference scheme, we thus have the following approximation
\begin{equation*}
F(X_{t_{n+1}}) \simeq A_{t_n}'(\xi(X_{t_n})) + \frac{ \xi (X_{t_{n+1}}) - \xi(X_{t_n})  - \sqrt{2
\beta^{-1}}  \frac{\nabla \xi}{|\nabla \xi|} (X_{t_n}) \cdot (B_{t_{n+1}} - B_{t_n})}{\Delta t}.
\end{equation*}
\end{remark}

\subsection{A PDE formulation and presentation of the main result}\label{sec:pres_res}

We would like to emphasize that our arguments are
partially {\em formal}: we assume that we are given a process $X_t$ and a function
$A_t'$ which satisfy~(\ref{eq:X})--(\ref{eq:A_prime_t}), and such that
$X_t$ has a smooth density $\psi(t,\cdot)$ with respect to the Lebesgue
measure on ${\mathcal D}$. We suppose that this density is 
sufficiently regular so that the computations are valid. In particular, we
assume that the potential $V$ is such that either the stochastic process
$X_t$ lives in~${\mathcal D}$ and thus that
its density $\psi(t,\cdot)$ decays sufficiently fast on $\partial
{\mathcal D}$ or the stochastic process $X_t$ has some reflecting
behavior on $\partial {\mathcal D}$ and thus that
its density $\psi(t,\cdot)$ has zero normal derivatives on $\partial
 {\mathcal D}$. In both cases, no boundary terms appear in the integrations by
parts we perform to derive the entropy estimates. We refer for example
to~\cite{arnold-markowich-toscani-unterreiter-01} for an appropriate
functional framework in which such entropy estimates hold.

Since only the law of the process $X_t$ at a fixed time $t$ is used
in~(\ref{eq:A_prime_t}), it is possible to recast the dynamics in the
following nonlinear partial differential equation (PDE) on the density $\psi(t,\cdot)$ of $X_t$:
\begin{equation}\label{eq:EDP}
\boxed{
\left\{
\begin{array}{l}
\dps{\partial_t \psi=\div\left(|\nabla \xi|^{-2} \left(\nabla
      (V-A_t\circ \xi + W \circ \xi) \psi + \beta^{-1}
  \nabla \psi \right)\right),}\\
A_t'(z)= \frac{ \dps{\int_{\Sigma_z} F |\nabla \xi|^{-1} \psi(t,\cdot) d\sigma_{\Sigma_z}}}{\dps{\int_{\Sigma_z} |\nabla \xi|^{-1} \psi(t,\cdot)  d\sigma_{\Sigma_z}}},
\end{array}
\right.
}
\end{equation}
where $F$ is defined by~(\ref{eq:F}).
This is obtained by using the fact that if $X_t$ has law
$\psi(t,x)\,dx$, then the law of $\xi(X_t)$ is $\psi^\xi(t,z)\,dz$
with
\begin{equation}\label{eq:psi_xi}
\psi^\xi(t,z)=\int_{\Sigma_z} |\nabla \xi|^{-1} \psi(t,\cdot)
d\sigma_{\Sigma_z},
\end{equation}
and the conditional law of $X_t$ with respect to
$\xi(X_t)=z$ is $\mu_{t,z}$ defined by
\begin{equation}\label{eq:mu_t_z}
d\mu_{t,z}=\frac{\psi(t,\cdot) |\nabla \xi|^{-1} d
  \sigma_{\Sigma_z}}{\psi^\xi(t,z)}.
\end{equation}
The probability measure $\psi^\xi(t,z)\,dz$ is
the image of the probability measure $\psi(t,x)\, dx$  by $\xi$. These expressions can be obtained using the co-area formula (see Appendix~\ref{sec:co-area}).

Before presenting the results, we would like to motivate
the introduction of this dynamics by the following formal observation. If the
potential~$A_t$ and the law of~$X_t$ reach a stationary state, then,
from the dynamics~(\ref{eq:X}) on $X_t$ (or from the partial differential
equation~(\ref{eq:EDP}) satisfied
by the distribution of $X_t$), we observe that this stationary law is proportional to $\exp(-\beta (V(x)-A_\infty
\circ \xi(x) + W \circ \xi(x)))\, dx$, where $A_\infty$ denotes the stationary state for
$A_t$ (this requires a uniqueness result for the law of~$X_t$, which holds for
example if
$|\nabla \xi|$ is uniformly bounded from below by a positive
constant). Then, from the definition~(\ref{eq:A_prime_t}) of the biasing
force, we obtain that,
necessarily, $A_\infty'=A'$ (where $A'$ is the mean force defined by~(\ref{eq:A_prime})). This proves the uniqueness of the stationary
state for this dynamics. We
can thus expect that $A_t'$ converges to the mean force $A'$ in the
longtime limit.

The interest  of the dynamics~(\ref{eq:X})--(\ref{eq:A_prime_t})  is actually twofold. First,  as expected
from the formal argument above,  in the longtime limit, $A_t'$ converges
to the mean force $A'$ defined by~(\ref{eq:A_prime}) (see
Equation~(\ref{eq:CV_MF_2}) below).  Second, using the
ABF method, the law of $\xi(X_t)$ has a simple diffusive behavior (see
Equation~(\ref{eq:EDP_psi_xi}) below).  The
metastable feature  of the  simple dynamics~(\ref{eq:Q}) along  $\xi$ is
thus corrected by the addition of the adaptive potential $A_t$.  The aim
of  this   paper  is  to  give   a  precise  statement   for  these  two
assertions, which are mathematical formalizations of the two main
characteristics  [S1] and [S2] of adaptive techniques mentioned in Section~\ref{sec:FED}. The proof of the longtime convergence relies on entropy
techniques, and  requires  appropriate assumptions on  the potentials
$V$, $W$ and the reaction coordinate $\xi$. We prove that under suitable
assumptions, the  convergence of $A_t'$  to $A'$ is  exponentially fast,
with a  rate of  convergence limited, at  the macroscopic level,  by the
rate of convergence of the law of $\xi(X_t)$ to its longtime limit,  and, at  the  microscopic level,  by the rate of convergence
to the equilibrium  conditioned probability measures $\mu_{\Sigma_z}$,
for all values $z$ of the reaction coordinate.

All these results are more precisely stated in Section~\ref{sec:res},
and the proofs are given in Section~\ref{sec:proof}. We would like to
mention that the main arguments of the proof are
given in a very simple case in Section~\ref{sec:proof_x} and that we
also present a result of convergence for the dynamics~(\ref{eq:original})--(\ref{eq:A_prime_t}) in Section~\ref{sec:original}.

\section{Precise statements of the results}\label{sec:res}

In Section~\ref{sec:entropy}, we recall some well-known results on
entropy and introduce the main notation used in the following to state
the convergence result. Section~\ref{sec:CV} is devoted to the
presentation of the convergence result for the
dynamics~(\ref{eq:X})--(\ref{eq:A_prime_t}). Finally, we give in
Section~\ref{sec:original} a (weaker) convergence result for the dynamics~(\ref{eq:original})--(\ref{eq:A_prime_t}).

\subsection{Entropy and Fisher information}\label{sec:entropy}

Let us consider $\psi$ and $A_t'$ which satisfy~(\ref{eq:EDP}) and let
 introduce the long-time limit of $\psi$,  $\psi^\xi$ (defined
by~(\ref{eq:psi_xi})) and $\mu_{t,z}$ (defined by~(\ref{eq:mu_t_z})):
$$\psi_\infty=(Z Z^\xi)^{-1}\exp(-\beta(V-A\circ \xi + W \circ \xi)),$$
$$\psi^\xi_\infty(z)=(Z^\xi)^{-1}\exp(-\beta W(z)),$$
$$d\mu_{\infty,z}=d\mu_{\Sigma_z}=Z_{\Sigma_z}^{-1}\exp(-\beta V) |\nabla \xi|^{-1} d
\sigma_{\Sigma_z},$$
where
$$Z^\xi=\int_{\mathcal M} \exp(-\beta W(z))\, dz.$$
We recall that
$$Z_{\Sigma_z}=\int_{\Sigma_z}  |\nabla \xi|^{-1} \exp(-\beta V) d
\sigma_{\Sigma_z},\, Z=\int_{{\mathcal D}} \exp(-\beta  V(x))  \,  dx.$$
Notice that $\int_{{\mathcal D}}  \psi_\infty=1$, and that the probability measure $\psi^\xi_\infty(z)\,dz$ is
the image of the probability measure $\psi_\infty(x)\, dx$  by $\xi$. 

In order to state the results, we also need to introduce the following
projection operators. For any $x \in {\mathcal D}$, we denote by
 $$P(x)=\I-\frac{\nabla \xi \otimes \nabla
  \xi}{|\nabla \xi|^2}(x)$$
the orthogonal projection operator onto the tangent space $T_x
\Sigma_{\xi(x)}$ to $\Sigma_{\xi(x)}$ at point~$x$, and by
$$Q(x)=\frac{\nabla \xi \otimes \nabla
  \xi}{|\nabla \xi|^2}(x)$$
the orthogonal projection operator onto the normal space $N_x
\Sigma_{\xi(x)}$ to $\Sigma_{\xi(x)}$ at point~$x$. We denote by
$\otimes$ the tensor product: For two vectors $u,v \in {\mathcal D}$, $u
\otimes v$ is a $n \times n$ matrix with components $(u
\otimes v)_{i,j}=u_i v_j$.

We measure the ``distance'' between $\psi$ (respectively
$\psi^\xi$) and $\psi_\infty$ (respectively
$\psi^\xi_\infty$) using the relative entropy $H(\psi |
\psi_\infty)$ (respectively $H(\psi^\xi |
\psi^\xi_\infty)$), where, for any two probability measures $\mu$
and $\nu$ such that $\mu$ is absolutely continuous with respect to $\nu$
(this property being denoted $\mu \ll \nu$ in the following),
$$H(\mu | \nu)=\int \ln \left(\frac{d\mu}{d\nu} \right) d \mu.$$
We recall the Csiszar-Kullback inequality:
\begin{equation}\label{eq:CK}
\|\mu-\nu \|_{TV} \leq \sqrt{2 H(\mu | \nu)}
\end{equation}
where $\|\mu-\nu\|_{TV}=\sup_{f,\, \|f\|_{L^\infty}\leq 1} \{ \int f
d(\mu - \nu) \}$ denotes the total variation norm of the signed measure
$\mu-\nu$. When $\mu$ and $\nu$ both have densities with respect to the
Lebesgue measure, $\|\mu-\nu\|_{TV}$ is simply the $L^1$ norm of the
difference between the two densities.

We denote the {\em total entropy} by $$E(t)=H(\psi(t,\cdot) | \psi_\infty),$$
the {\em macroscopic entropy} by $$E_M(t)=H(\psi^\xi(t,\cdot) |
\psi^\xi_\infty),$$ the ``local entropy'' at a
fixed value $z$ of the reaction coordinate by
$$e_m(t,z)=H( \mu_{t,z} | \mu_{\infty,z} )= \int_{\Sigma_z} \ln\left( \frac{\psi(t,\cdot)}{
\psi^\xi(t,z)} \Big/ \frac{\psi_\infty }{
\psi^\xi_\infty(z)}\right) \frac{\psi(t,\cdot) |\nabla \xi|^{-1} d
\sigma_{\Sigma_z}}{\psi^\xi(t,z)},$$
and the {\em microscopic entropy} by
$$E_{m}(t)=\int_{\mathcal M} e_m(t,z)
\psi^\xi(t,z) \, dz.$$
It is straightforward to obtain the
following result which can be seen as the extensivity of the
entropy:
\begin{lemma}\label{lem:ext_ent}
It holds $$E(t)=E_M(t)+E_{m}(t).$$
\end{lemma}
Let us now introduce the Fisher information: For any two probability measures $\mu$
and~$\nu$ such that $\mu \ll \nu$, 
\begin{equation}\label{eq:fisher}
I(\mu | \nu)=\int
\left|\nabla \ln \left(\frac{d\mu}{d\nu} \right)\right|^2 d \mu.
\end{equation}
In the case $\nu$ is a probability measure on the (Riemannian) submanifold
$\Sigma_z$, $\nabla$ actually denotes the gradient on~$\Sigma_z$
in~(\ref{eq:fisher}), namely
\begin{equation}\label{eq:grad_surf}
\nabla_{\Sigma_z}=P \nabla.
\end{equation}
Therefore, for the conditional probability measures $\mu_{t,z}$ and
$\mu_{\infty,z}$, the Fisher information writes
$$I(\mu_{t,z} | \mu_{\infty,z})=\int_{\Sigma_z}
\left|\nabla_{\Sigma_z} \ln\left( \frac{\psi(t,\cdot)}{\psi_\infty}
\right)\right|^2 \frac{\psi(t,\cdot) |\nabla
\xi|^{-1} d \sigma_{\Sigma_z}}{\psi^\xi(t,z)}.$$
%%%%%%%%%%%%%%%%%%%%%%%%%%%%%%%%%%%%%%%%%%%%%%%%%%%%%%%%%%%%
%since $\nabla_{\Sigma_z}
%\frac{\psi^\xi}{\psi^\xi_\infty}=0$.
%%%%%%%%%%%%%%%%%%%%%%%%%%%%%%%%%%%%%%%%%%%%%%%%%%%%%%%%%%%%
Let us finally introduce another way to compare two probability
measures, namely the  Wasserstein distance with quadratic cost: for
two probability measures $\mu$ and $\nu$ defined on a Riemannian manifold $\Sigma$,
$$W(\mu,\nu)=\sqrt{\inf_{\pi \in \Pi(\mu,\nu)} \int_{\Sigma \times
    \Sigma} d_\Sigma(x,y)^2 \, d\pi(x,y)}.$$
In this expression, $d_\Sigma$ denotes the geodesic distance on $\Sigma$: $\forall x,y \in \Sigma$,
$$d_\Sigma(x,y)=\inf \left\{ \sqrt{\int_0^1 |\dot{w}(t)|^2 \, dt} \,
  \Bigg| \, w \in
  {\mathcal C}^1([0,1],\Sigma),\, w(0)=x,\, w(1)=y \right\},$$
where $\Pi(\mu,\nu)$
denotes the set of coupling probability measures, namely probability measures on
$\Sigma \times \Sigma$ such that their marginals are $\mu$ and $\nu$.
We need the following definitions:
\begin{definition}\label{def:LSI}
The probability measure $\nu$ is said to satisfy a logarithmic Sobolev inequality
with constant $\rho>0$ (in short: LSI($\rho$)) if for all probability
measures $\mu$ such that $\mu \ll \nu$,
$$H(\mu|\nu) \leq \frac{1}{2 \rho} I (\mu | \nu).$$
\end{definition}
\begin{definition}
The probability measure $\nu$ is said to satisfy a Talagrand inequality 
with constant $\rho>0$ (in short: T($\rho$)) if for all probability
measures $\mu$ such that $\mu \ll \nu$,
$$W(\mu,\nu) \leq \sqrt{\frac{2}{\rho} H (\mu | \nu)}.$$
\end{definition}
In the latter definition, we implicitly assume that the probability measures have
finite moments of order~2. This will always be the case for all the
probability measures we consider. We will need the following important
result (see~\cite[Theorem 1]{otto-villani-00}).
\begin{lemma}\label{lem:otto_villani}
If $\nu$ satisfies LSI($\rho$), then $\nu$ satisfies T($\rho$).
\end{lemma}

For an introduction to logarithmic Sobolev inequalities, their
properties and their relation to longtime behavior of solutions to PDEs, we refer
to~\cite{ABC-00,arnold-markowich-toscani-unterreiter-01,villani-03}.

\subsection{Convergence of the adaptive dynamics~(\ref{eq:X})--(\ref{eq:A_prime_t})}\label{sec:CV}

We are now in position to state our main results. Concerning the
dynamics on the law of $\xi(X_t)$, we have:
\begin{proposition}[Equation satisfied by the marginal density $\psi^\xi$]\label{prop:psi_xi}
Let $(\psi,A_t')$ be a smooth solution to~(\ref{eq:EDP}) and let us
assume~[H1]. Then $\psi^\xi$ satisfies the following equation:
\begin{equation}\label{eq:EDP_psi_xi}
\partial_t \psi^\xi =  \partial_{z} \left(W'  \psi^\xi + \beta^{-1}
  \partial_{z}\psi^\xi  \right)
\text{ on ${\mathcal M}$.}
\end{equation}
\end{proposition}

\begin{remark}
Notice that even if $\psi^\xi$ satisfies a closed PDE, $\xi(X_t)$ does
not satisfy a closed SDE (see Equation~(\ref{eq:SDE_xi}) above).
\end{remark}

The fundamental assumptions we need to prove longtime convergence are
the following (we recall that the local mean force $F$ is defined by~(\ref{eq:F})):
\begin{equation*}\label{eq:hyp_V}
\text{{\bf [H2]}~~~}
\left\{
\begin{array}{c}
\text{$V$ and $\xi$ are sufficiently differentiable functions such that} \\
\text{$\|\nabla \xi\|_{L^\infty} \leq m<\infty$ and $\left\|\nabla_{\Sigma_z} F  \right\|_{L^\infty} \leq M<\infty$, }
\end{array}
\right.
\end{equation*}
\begin{equation*}\label{eq:hyp_LSI}
\text{{\bf [H3]}~~~}
\left\{
\begin{array}{c}
\text{$V$ and $\xi$ are such that $\exists \rho >0$, for all $z \in
  {\mathcal M}$,}\\
\text{ the conditional measure $\mu_{\infty,z}$ satisfies LSI($\rho$).}
\end{array}
\right.
\end{equation*}
In Assumption [H2], the requirement on $F$ can be seen as a boundedness
condition on the coupling between the conditional measures
$\mu_{\infty,z}$ and the corresponding marginal~$\psi^\xi_\infty$, since it involves the mixed
derivatives (along the tangential space and the normal
space of the submanifold $\Sigma_z$) $P \nabla (Q \nabla V)$ (see~\cite{otto-reznikoff-07} and
Remark~\ref{rem:coupling} below).

 Assumption [H3] ensures that if, for a fixed value $z$ of the reaction coordinate, the
conditioned probability measure $\mu_{\infty,z}$ were to be sampled by a
simple constrained gradient dynamics
(see~\cite{ciccotti-lelievre-vanden-einjden-06}), the 
convergence to equilibrium would be exponential with rate $\rho$.
We refer to $\rho$ as the {\em microscopic rate of convergence} in the sequel.

We refer  to Section~\ref{sec:proof_x}  for an explicit  framework where
[H2] and [H3] are satisfied, and to Remark~\ref{rem:hyp_V_xi} below for alternative
assumptions on $V$ and $\xi$.

Let us now introduce the assumption we need on $W$.
\begin{equation*}\label{eq:hyp_W}
\text{{\bf [H4]}~~~$W$ is such that $\exists I_0>0, r>0$, $\forall t \geq 0$, $I(\psi^\xi(t,\cdot)|\psi^\xi_\infty) \leq
  I_0 \exp(-2 \beta^{-1} \,r \,t)$.}
\end{equation*}
Assumption~[H4] is indeed an assumption on $W$ because $\psi^\xi$ satisfies the
PDE~(\ref{eq:EDP_psi_xi}) where only $W$ appears.  Assumption [H4] ensures that the law of $\xi(X_t)$ converges
to equilibrium exponentially fast with rate $r$,  which
we refer to as the {\em macroscopic rate of convergence} in the sequel.

We will see below
(see [H4']) some sufficient explicit conditions on $W$ for [H4] to be
satisfied.

\begin{theorem}[Exponential convergence of the entropy to
  zero]\label{theo:CV}

Let us assume [H1],~[H2],~[H3]
and~[H4]. Then the microscopic entropy $E_m$ satisfies:
\begin{equation}\label{eq:CV_mic_ent}
\sqrt{E_m(t)} \leq C \exp(-\lambda t)
\end{equation}
where $C=2 \max\left(\sqrt{E_m(0)},\frac{  M }{\beta^{-1}|\rho m^{-2} - r|}
\sqrt{\frac{I_0}{2 \rho}} \right)$
and 
\begin{equation}\label{eq:lambda}
\lambda=\beta^{-1} \min(\rho m^{-2},r).
\end{equation}
In the special case $\rho m^{-2}= r$, $E_m$ satisfies $\sqrt{E_m(t)} \leq  \left(
  \sqrt{E_m(0)} + M \sqrt{\frac{I_0}{2
      \rho}} \, t  \right) \exp(- \beta^{-1} r \, t)$.

This implies that the
total entropy $E$ and thus $\|\psi(t,\cdot) -
\psi_\infty\|_{L^1({\mathcal D})}$ both converge exponentially fast to zero with rate $\lambda$.

We thus obtain that the biasing force $A'_t$ converges to the mean
  force $A'$ in the following sense: $\forall t \geq 0$,
\begin{equation}\label{eq:CV_MF_1}
\int_{{\mathcal M}} |A_t' - A'|^2(z) \psi^\xi(t,z) \, dz \leq \frac{2
   M^2}{\rho}  E_m(t).
\end{equation}
\end{theorem}
Notice that the fact that $E$ and $\|\psi(t,\cdot) -
\psi_\infty\|_{L^1({\mathcal D})}$ converge exponentially fast to zero
with rate $\lambda$ is an immediate consequence
of~(\ref{eq:CV_mic_ent}),~[H4],~Lemma~\ref{lem:ext_ent} and the Csiszar-Kullback inequality~(\ref{eq:CK}).

 We will actually
consider the two following cases for
which~[H4] is satisfied:
\begin{equation*}\label{eq:hyp_W_2}
\text{{\bf [H4']}~~~}\left\{
\begin{array}{ll}
\text{If ${\mathcal M}=\T$},& \text{then $W=0$.}\\
\text{If ${\mathcal  M}=\R$},& \text{then $W$ is a potential such that $W''$ is
  bounded from below}\\
& \text{and there exists $\overline{r}>0$ such that  $\frac{\exp(-\beta W)}{\int_{{\mathcal M}}  \exp(-\beta W)}$ satisfies LSI($\overline{r}$).}
\end{array}
\right.
\end{equation*}
 Notice that in the case ${\mathcal  M}=\R$, the assumptions stated
 in~[H4'] on $W$ are satisfied for an $\alpha$-convex potential (namely
 if $W'' \geq \alpha$ for a positive $\alpha$), and
 then it is possible to choose $r=\alpha$ in [H4] (see
 Lemma~\ref{lem:fisher_macro_R} below).
We refer to Remark~\ref{rem:hyp_W} below for alternative assumptions on $W$.

%  We recall that a differentiable function $W:\R^n  \to \R$ is
%  $\alpha$-convex (for a positive~$\alpha$) if and only if, for all $x, y
%  \in \R^n$,
% $$W(y) \geq W(x) + \nabla W(x) \cdot (y-x) + \frac{\alpha}{2} |y-x|^2.$$

\begin{corollary}[Convergence of the biasing force]\label{cor:W}
If [H4'] is satisfied and $\psi^\xi$ satisfies~(\ref{eq:EDP_psi_xi}) then [H4] holds.

More precisely, if ${\mathcal M}=\T$ and $W=0$, then [H4] is satisfied with $I_0=I(\psi^\xi(0,\cdot)|\psi^\xi_\infty)$ and $r=4
\pi^2$. If ${\mathcal  M}=\R$, $W''$ is
  bounded from below and $\frac{\exp(-\beta W)}{\int_{{\mathcal M}} \exp(-\beta W)}$
  satisfies LSI($\overline{r}$), then [H4] is satisfied with
  $r=\overline{r}-\varepsilon$ for any $\varepsilon \in (0,
\overline{r})$.

Let us now assume~[H1],~[H2],~[H3] and~[H4']. From~(\ref{eq:CV_MF_1}), we deduce  that for all compact $K  \subset {\mathcal M}$, $\exists
\overline{C}, t^*>0$, $\forall t \geq t^*$,
\begin{equation}\label{eq:CV_MF_2}
\int_K |A_t' - A'|(z) \psi^\xi_\infty(z) \, dz \leq \overline{C}
\exp(-\lambda t),
\end{equation}
where $\lambda$ is the rate of convergence defined by~(\ref{eq:lambda}) in Theorem~\ref{theo:CV}.
\end{corollary}

%%%%%%%%%%%%%%%%%%%%%%%%%%%%%%%%%%%%%%%%%%%%%%%%%%%%%%%%%%%%%%%%%%%%%%%%%%%
% Rque : $\rho$ et $r$ dépendent aussi de $\beta$...
%%%%%%%%%%%%%%%%%%%%%%%%%%%%%%%%%%%%%%%%%%%%%%%%%%%%%%%%%%%%%%%%%%%%%%%%%%%

  These results therefore show  that $A_t'$ converges exponentially fast
to   $A'$  (in \linebreak   $L^1(\psi^\xi_\infty(z)   \,  dz)$-norm)   at  a   rate
$\lambda=\beta^{-1}  \min(\rho m^{-2},r)$. The  limitations on  the rate
$\lambda$ are related to the  rate of convergence $r$ at the macroscopic
level, for  the equation~(\ref{eq:EDP_psi_xi}) satisfied  by $\psi^\xi$,
and the rate  of convergence at the microscopic  level, which depends on
the constant $\rho$ of the logarithmic Sobolev inequalities satisfied by
the  conditional  measures $\mu_{\infty,z}$.  This  constant of
course depends   on  the choice of the  reaction coordinate. In  our framework, we
could state that  a ``good reaction coordinate'' is  such that $\rho$ is
as large as possible. 

The  proof  of these  results  is  given in  Sections~\ref{sec:proof_x},
~\ref{sec:proof_gen} and~\ref{sec:proof_cor} below. 

\begin{remark}[Other possible assumptions on $V$ and $\xi$]\label{rem:hyp_V_xi}
We would like to mention other possible assumptions on $V$ and $\xi$
than [H2]--[H3] for
which the results of Theorem~\ref{theo:CV} still hold.

\begin{itemize}
\item
First, in [H2], it is possible to change the assumption
$\left\|\nabla_{\Sigma_z} F  \right\|_{L^\infty} \leq M<\infty$ to
$$\left\| F  \right\|_{L^\infty} \leq M<\infty.$$
Indeed, this simply changes the estimate~(\ref{eq:1}) in Lemma~\ref{lem:estim_1} below to the
following \begin{align*}
|A_t'(z)-A'(z)| & \leq \left\| F  \right\|_{L^\infty} \|\mu_{t,z} - \mu_{\infty,z} \|_{TV},\\
& \leq M \sqrt{2 H(\mu_{t,z}|\mu_{\infty,z})},
\end{align*}
by the Csiszar-Kullback inequality~(\ref{eq:CK}). The rest of the proof
remains exactly the same.
%%%%%%%%%%%%%%%%%%%%%%%%%%%%%%%%%%%%%%%%%%%%%%%%%%%%%
% Ca simplifie aussi le preuve de convergence de A_t' en fait.
% cf. ``CV de A_t' sous l'hypothèse F borné'' ci dessous
%%%%%%%%%%%%%%%%%%%%%%%%%%%%%%%%%%%%%%%%%%%%%%%%%%%%%
\item Second, it is possible to obtain a similar result of convergence under slightly different
  assumptions than [H2]--[H3] by introducing another Riemannian
  structure on the submanifolds $\Sigma_z$. This is made precise in
  Appendix~\ref{sec:rieman} (see assumptions [H2']--[H3']).
\end{itemize}

\end{remark}

\begin{remark}[Other possible assumptions on $W$]\label{rem:hyp_W}
From Lemma~\ref{lem:fisher_macro_T} and~\ref{lem:fisher_macro_R} below (used
to prove Corollary~\ref{cor:W}), it will become clear that~[H4] is
actually satisfied with $W=0$ as soon as ${\mathcal M}$ is a bounded domain. If
${\mathcal M}$ is an unbounded domain, then a potential~$W$ with
properties such as those stated in [H4'] is needed. We discuss in this
remark other properties on $W$ to satisfy [H4] than those proposed in [H4'], in
the case ${\mathcal M}=\R$ (or ${\mathcal M}$ is an unbounded domain).

In this case, it is actually also possible to
satisfy~[H4] by choosing $W$ such that the dynamics is confined in a
domain $\bigcup_{z \in {\mathcal N}} \Sigma_z$, where ${\mathcal N}$ is
a bounded subset of ${\mathcal M}$. This can be done by using a
sufficiently confining potential $W$ and adapting Lemma~\ref{lem:fisher_macro_R} below, or by adding reflexion
terms to restrict $\xi$ to ${\mathcal N}$ (which loosely
speaking corresponds to take $W$ zero on ${\mathcal N}$ and infinite on
${\mathcal M}\setminus{\mathcal N}$) and adapting Lemma~\ref{lem:fisher_macro_T} below.

Let us make precise this latter case. Suppose for example we are interested in the values of
$A'(z)$ for $z \in {\mathcal N}=(0,1)$. The dynamics is confined in the domain ${\mathcal
O}=\bigcup_{0 < z < 1} \Sigma_z$. The ABF dynamics is 
\begin{equation*}
\left\{
\begin{array}{ll}
\dps{\partial_t \psi=\div\Big(|\nabla \xi|^{-2} \left(\nabla (V-A_t\circ \xi) \psi + \beta^{-1}
  \nabla \psi \right)\Big),} & \text{{\rm  on} ${\mathcal
O}$},\\
\dps{\left(\nabla (V-A_t\circ \xi) \psi + \beta^{-1}
  \nabla \psi \right) \cdot \nabla \xi = 0,} &\text{{\rm on} $\Sigma_0 \cup \Sigma_1$},\\
A_t'(z)= \frac{\dps{\int_{\Sigma_z} F |\nabla \xi|^{-1} \psi(t,\cdot) d\sigma_{\Sigma_z}}}{\dps{ \int_{\Sigma_z} |\nabla \xi|^{-1}
    \psi(t,\cdot)  d\sigma_{\Sigma_z}}}, & \text{{\rm for} $z \in (0,1)$},
\end{array}
\right.
\end{equation*}
where $F$ is defined by~(\ref{eq:F}).
From the point of view of the stochastic process $X_t$, the boundary condition
 translates to a normal reflexion
on the two submanifolds~$\Sigma_0$ and~$\Sigma_1$. Moreover, it can be
checked (using Lemma~\ref{lem:psi_bar_prime}) that the boundary condition on~$\psi$ translates to a zero Neumann
boundary condition on $\psi^\xi$: $\partial_z
\psi^\xi(0)=\partial_z \psi^\xi(1)=0$. A proof similar to that of
Lemma~\ref{lem:fisher_macro_T} then shows that $I(\psi^\xi|\psi^\xi_\infty)$ converges
  exponentially fast to $0$, so that~[H4] holds.
The arguments we use to prove Theorem~\ref{theo:CV} and Corollary~\ref{cor:W}
then show that $\|A_t' -A' \|_{L^2(0,1)}$ goes to $0$
exponentially fast.
\end{remark}

\begin{remark}[Vectorial reaction coordinate]\label{rem:xi_multi_dim}
In this work, we assume that the reaction coordinate $\xi$ has values
in $\T$ or $\R$. The
dynamics~(\ref{eq:X})--(\ref{eq:A_prime_t}) and the 
results of convergence presented in this section can be straightforwardly extended to the case when $\xi=(\xi_1,\ldots,\xi_m)$ has values in $\T^m$
or $\R^m$,
with $2 \leq m < n$, under the orthogonality condition:
\begin{equation}\label{eq:ortho}
\forall i \neq j, \, \nabla \xi_i \cdot \nabla \xi_j=0.
\end{equation}
The generalization of this dynamics to non orthogonal reaction coordinates is unclear. In
this case, it is possible to resort to metadynamics (see Remark~\ref{rem:meta}
below). Alternatively, the
dynamics~(\ref{eq:original})--(\ref{eq:A_prime_t}) (and the result of
convergence of Section~\ref{sec:original} for this dynamics) can straightforwardly be generalized to a vectorial reaction
coordinate. 
\end{remark}
% Le problème est que je ne sais pas écrire une dynamique simple qui se
% réduit à une diffusion sur la marginale.

\begin{remark}[Metadynamics]\label{rem:meta}
The adaptive biasing force technique can also be used in the context of
metadynamics~\cite{iannuzzi-laio-parrinello-03,bussi-laio-parinello-06,lelievre-rousset-stoltz-07-b}.
The principle of metadynamics is to introduce an additional variable $z$
with dimension the dimension of $\xi$ (say $z \in \R^m$, with $1 \leq m < n$), and
an extended potential $V_\zeta(q,z)=V(q)+ \frac{\zeta}{2}|z -\xi(q)|^2$. The
reaction coordinate is then chosen to be $\xi_{\rm meta} (q,z)=z$ so
that the associated free energy is $$A_\zeta(z) = - \beta^{-1} \ln \int_{{\mathcal D}} 
\exp(-\beta V_\zeta(q,z) ) \, dq,$$ which converges to $A(z)$ when $\zeta$
goes to infinity. In our framework, the ABF method applied to this extended system 
writes:
\begin{equation*}
\left\{
\begin{array}{l}
dX_t= \left( -  \nabla V(X_t) + \zeta(Z_t - \xi(X_t)) \nabla \xi(X_t) \right)
\,dt + \sqrt{2 \beta^{-1}}   dB_t,\\
dZ_t=\zeta \Big(  \xi(X_t) - \E(\xi(X_t) | Z_t) \Big) \, dt + \sqrt{2 \beta^{-1}} d\overline{B}_t,
\end{array}
\right.
\end{equation*}
where $\overline{B}_t$ is a $m$-dimensional Brownian motion, independent of
$B_t$. Notice that by construction, the orthogonality
condition~\eqref{eq:ortho} is satisfied by $\xi_{\rm meta}$, so that the
convergence results of this section apply to these kinds of models.
\end{remark}

\begin{remark}[On the initial condition]
If $\psi^\xi(0,\cdot)$ is zero at some points or is not sufficiently
smooth, then $A_0'$ may be not well defined or $I(\psi^\xi(0,\cdot) |
\psi^\xi_\infty)$ may be infinite (which is in contradiction with~[H4]). But since we show that
$\psi^\xi$ satisfies a simple diffusion equation (see Proposition~\ref{prop:psi_xi}), these difficulties
disappear as soon as $t>0$. Therefore, up to considering the problem for
$t \geq t_* >0$, we can suppose that $\psi^\xi(0,\cdot)>0$.
\end{remark}

\begin{remark}[On the choice of the entropy]\label{rem:entropy}
In the case of linear Fokker Planck equations, it is well known that one
can obtain exponential convergence to equilibrium by considering various
entropies of the form $\int h\left( \frac{d \mu}{d \nu}\right)
\, d \mu$, where $h$ is
typically a strictly convex function such that $h(1)=0$
(see~\cite{arnold-markowich-toscani-unterreiter-01} for more assumptions
required on $h$). For example, the classical choice $h(x)=\frac{1}{2} (x-1)^2$ is linked to Poincar\'e type inequalities and leads to $L^2$-convergence, while the
function $h(x) = x \ln x - x + 1$ we have used here to build the entropy is linked to
logarithmic Sobolev inequalities and leads to $L^1 \ln
L^1$-convergence. However, for the study of the non-linear Fokker Planck
equation~(\ref{eq:EDP}), it seems that the choice $h(x) = x \ln x - x + 1$ is necessary
to derive the estimates, for example to have the extensivity property of Lemma~\ref{lem:ext_ent}.
\end{remark}

\begin{remark}[Smoother evolution in time of $A_t'$]
In practice, it may be useful to
update the adaptive potential $A_t'$ in a smoother way in time, for
example by replacing~(\ref{eq:A_prime_t}) by
$$dA_t'(z)=\frac{1}{\tau}\left(\E\left(F (X_t) \, \Big| \xi(X_t)=z \right)
- A_t'(z) \right)\, dt,$$ where $F$ is defined by~(\ref{eq:F}) and $\tau>0$ denotes a characteristic time
(possibly depending on $(t,z)$), to
be fixed. This amounts to replace $A_t'$ by $\kappa_\tau * A_t'$
in~(\ref{eq:X}), where $\kappa_\tau$ is an exponential convolution kernel. Formally, we here consider the limit case $\tau=0$. To prove
the convergence of $A_t'$ towards $A'$ for $\tau \neq 0$ is an open problem.
\end{remark}

\begin{remark}[Enhancing the macroscopic rate of convergence]\label{rem:variance}
Let us consider the case ${\mathcal M}=\R$. For an $\alpha$-convex
potential $W$, Corollary~\ref{cor:W} states
that $A_t'$ converges towards $A'$ exponentially fast, with a
rate $\lambda=\beta^{-1} \min ( \rho m^{-2}, \alpha)$. This may seem
surprising since for large enough $\alpha$, the rate of convergence is
no more limited by~$\alpha$. However, it is typically expected that the
constant $I_0$ in assumption [H4] increases with growing $\alpha$, which
means that the constant $C$ increases in the convergence estimate~(\ref{eq:CV_mic_ent}). Moreover, in practice, if $\alpha$ is very
large, $\psi_\infty^\xi$ is very peaked and some parts of ${\mathcal M}$ are poorly sampled, so
that the variance of the result is large in these areas (which can
not be seen in our convergence result). Actually, a good method to
enhance the rate of convergence at the macroscopic level while keeping
a good sampling and thus low variance, is to use a particle systems with
many replicas and a selection mechanism. We refer
to~\cite{lelievre-rousset-stoltz-07-b} for more details.
\end{remark}

\subsection{A convergence result for the adaptive dynamics~(\ref{eq:original})--(\ref{eq:A_prime_t})}\label{sec:original}

In this section, we present a weaker convergence result for another
adaptive overdamped Langevin dynamics, namely (\ref{eq:original})--(\ref{eq:A_prime_t}). For simplicity, we only consider the case
$$\text{${\mathcal M}=\T$ and $W=0$,}$$ but the results can be extended to the case
${\mathcal M}=\R$ with a suitable $W \neq 0$, as in
Section~\ref{sec:CV} (see [H4] and [H4']). One interest of this
dynamics and this result of convergence
is that they can be straightforwardly extended to the case of a
multi-dimensional reaction coordinate (see Remark~\ref{rem:xi_multi_dim}
above). For the sake of conciseness, we
do not provide the details of the result in this case which follows
exactly the same lines (see~\cite{ciccotti-lelievre-vanden-einjden-06}
and Appendix~\ref{sec:co-area} for formulas in the case of a multi-dimensional reaction coordinate).
Let us recall the dynamics~(\ref{eq:original})--(\ref{eq:A_prime_t}) we
consider here:
\begin{equation}
\boxed{
\label{eq:X_ori}
dX_t= -  \nabla \Big(V  - A_t \circ \xi \Big)(X_t) \, 
\,dt + \sqrt{2 \beta^{-1}}  dB_t,
}
\end{equation}
with the same definition as before for $A_t$: $\forall z \in \T$,
\begin{equation}\label{eq:A_prime_t_ori}
\boxed{
  A_t'(z)=\E\left(F (X_t) \, \Big| \xi(X_t)=z \right),}
\end{equation}
where $F$ is defined by~(\ref{eq:F}).
The associated non-linear Fokker Planck equation is now:
\begin{equation}\label{eq:EDP_ori}
\boxed{
\left\{
\begin{array}{l}
\dps{\partial_t \psi=\div\left(\nabla (V-A_t\circ \xi) \psi + \beta^{-1}
  \nabla \psi \right),}\\
A_t'(z)=\frac{ \dps{\int_{\Sigma_z} F \, |\nabla \xi|^{-1} \psi(t,\cdot) d\sigma_{\Sigma_z}}}{ \dps{ \int_{\Sigma_z} |\nabla \xi|^{-1} \psi(t,\cdot)  d\sigma_{\Sigma_z}}}.
\end{array}
\right.
}
\end{equation} 

The main difference with the dynamics~(\ref{eq:X})--(\ref{eq:A_prime_t}) considered in Theorem~\ref{theo:CV} is that
the marginal  distribution $\psi^\xi$ does not satisfy  a closed partial
differential equation. Therefore, we do not know {\it a priori} that the
Fisher    information    $I(\psi^\xi|\psi^\xi_\infty)$   converges    to
$0$.  The strategy here is to  directly estimate  the derivative  of the
total  entropy $E$. We obtain a convergence result under two additional assumptions (see~[H5]--[H6]).

\begin{theorem}[Longtime convergence for the dynamics~(\ref{eq:original})--(\ref{eq:A_prime_t})]\label{theo:CV_ori}
Let $(\psi,A_t')$ be a smooth solution to~(\ref{eq:EDP_ori}) and let us
assume~[H1],~[H2],~[H3]. Moreover, we suppose
\begin{equation}\label{eq:hyp_psi_infty}
\text{{\bf [H5]}~~~$V$ and $\xi$ are such that $\exists R>0$, $\psi_\infty$ satisfies LSI($R$),}
\end{equation}
and
\begin{equation*}\label{eq:hyp_mMrho}
\text{{\bf [H6]}~~~}\frac{mM \beta}{2\sqrt{\rho}} < 1.
\end{equation*}
Then the total entropy $E$ satisfies:
$$\sqrt{E(t)} \leq \sqrt{E(0)} \exp(-\lambda t)$$
where $\lambda=\beta^{-1}\left(-1+\frac{mM\beta}{2\sqrt{\rho}}\right) R$
is positive using [H6]. In
  particular, as in Theorem~\ref{theo:CV}, the biasing force $A_t'$ converges
  exponentially fast to the mean force $A'$.
\end{theorem}

%%%%%%%%%%%%%%%%%%%%%%%%%%%%%%%%%%%%%%%%%%%%%%%%%%%%%%%%%%%%%%%%%
% Ce n'est pas du tout clair que [H6] is satisfied if the temperature is
% sufficiently large ({\em i.e.} $\beta$ is small enough).
%
% Il faut voir le comportement des paramètres M et rho en fonction
% de beta. Notamment pour rho, c'est compliqué...
%%%%%%%%%%%%%%%%%%%%%%%%%%%%%%%%%%%%%%%%%%%%%%%%%%%%%%%%%%%%%%%%%

The  proof  of this  result  is  given in  Section~\ref{sec:proof_original} below. 

\begin{remark}[On assumption~{[H5]}]\label{rem:coupling}
In~\cite[Theorem 2]{otto-reznikoff-07}, it is shown that if
$\mu=\exp(-H(x_1,x_2)) dx_1 dx_2$ is a
probability measure on a product space $X=X_1 \times X_2$ (where $X_i$ are
Euclidean spaces), if the conditional probabilities $\mu(dx_2 |x_1)$ satisfy
LSI($\rho_2$) (with $\rho_2$ independent of $x_1$) and the marginal
$\overline{\mu}(dx_1)$ satisfies LSI($\overline{\rho_1}$), then $\mu$
satisfies LSI($\rho$) provided the coupling between the two directions
is bounded: $\exists \kappa_{1,2}>0$, $\forall (x_1,x_2) \in
X_1 \times X_2$,
$$\left|\partial^2_{x_1,x_2} H(x_1,x_2)\right| \leq \kappa_{1,2}.$$
Thus, in the simple framework of Section~\ref{sec:proof_x} for example, where the
configuration space is $\T \times \R$ and the reaction coordinate is
$\xi(x,y)=x$, the fact that $\psi_\infty$ satisfies a LSI
(assumption~[H5]) can be deduced from the fact that the
conditioned distributions~$\mu_{\infty,z}$ satisfy a LSI (which is~[H3]),
the marginal $\psi^\xi_\infty$ satisfy a LSI (which is related to~[H4]) and the coupling is bounded (which
is~[H2]). Thus~[H5] is not needed as an
additional assumption compared to the framework of
Theorem~\ref{theo:CV}. The generalization of this result to the case when $X$
is not a product does not seem to be straightforward.
\end{remark}

\section{Proofs}\label{sec:proof}

One remark to simplify the presentation of the proofs is that we can suppose
$\beta=1$ up to the following change of variable: $\tilde{t}=\beta^{-1} t$,
$\tilde{\psi}(\tilde{t},x)=\psi(t,x)$, $\tilde{V}(x)=\beta
V(x)$ and $\tilde{W}(x)=\beta W(x)$. Therefore, we suppose in the following that
\begin{equation}\label{eq:hyp_beta}
\beta=1.
\end{equation}
% Notice however that the constants $M$ and $\rho$ depends on $\beta$.

\subsection{Proof of Proposition~\ref{prop:psi_xi} and Theorem~\ref{theo:CV} in a simple case}\label{sec:proof_x}

In this section, we propose to prove Proposition~\ref{prop:psi_xi} and Theorem~\ref{theo:CV} in the simple
case $n=2$, $\xi(x,y)=x$ (so that we use in this section the notation $x$ instead of $z$ for the
reaction coordinate variable) and the configuration space is ${\mathcal D}=\T \times \R$
(which means that all the data are periodic with respect to the first
coordinate $x$). In this case, we thus have $\xi \in \T$ (${\mathcal
  M}=\T$) so that we
choose $W=0$ (see~[H4']). Notice also that the local mean force~$F$ is
simply given by $F=\partial_x V$ (see~(\ref{eq:F})). Our aim is to
introduce the main arguments in this simple case before presenting
the general proof in Section~\ref{sec:proof_gen}.

In this simple setting, the system~(\ref{eq:EDP}) writes (recall $\beta=1$):
\begin{equation}\label{eq:EDP_x}
\left\{
\begin{array}{l}
\dps{\partial_t \psi=\div\left(\nabla V \psi + 
  \nabla \psi \right) - \partial_x(A_t' \psi),}\\
\dps{A_t'(x)= \frac{\dps{\int_{\R} \partial_x V(x,y) \psi(t,x,y) dy}}{\dps{\psi^\xi(t,x)}},}
\end{array}
\right.
\end{equation}
where $\psi^\xi(t,x)=\int_{\R} \psi(t,x,y) dy$. Notice that in this case
$\psi^\xi_\infty \equiv 1$.

It can be checked that the assumptions~[H2] and
[H3] are satisfied in this context for a potential $V$ of
the following form:
$$V(x,y)=V_0(x,y) + V_1(x,y)$$
where $\inf_{\T \times \R} \partial_{y,y} V_0 > 0$, $\|V_1\|_{L^\infty} <
\infty$, $\|\partial_{x,y} (V_0+V_1)\|_{L^\infty} <
\infty$. The potential $V$ is thus a bounded perturbation of an
$\alpha$-convex potential, with a bounded mixed derivative
$\partial_{x,y} V$. Then, assumptions [H2]--[H3] are satisfied with $m=1$, $M=\|\partial_{x,y} V
\|_{L^\infty}$ and $\rho=\left(\inf_{\T \times \R} \partial_{y,y} V_0 \right) \, \exp(-
\text{osc } V_1)$, where $\text{osc } V_1=\sup_{\T \times \R} V_1 -
\inf_{\T \times \R} V_1$ (see~\cite{ABC-00}).

%%%%%%%%%%%%%%%%%%%%%%%%%%
% OK pour $V(x,y)=\cos(x)+(y-\cos(x))^2$ par exemple.
%%%%%%%%%%%%%%%%%%%%%%%%%%

Proposition~\ref{prop:psi_xi} is simply obtained by integration of~(\ref{eq:EDP_x}) with
respect to~$y \in \R$:
\begin{lemma}\label{lem:psi_bar_x}
The density $\psi^\xi$ satisfies the following equation on $\T$:
\begin{equation}\label{eq:psi_bar_x}
\partial_t \psi^\xi = \partial_{x,x} \psi^\xi.
\end{equation}
\end{lemma}

As stated in Corollary~\ref{cor:W}, this result already yields the exponential
convergence to zero of the macroscopic Fisher information
$I(\psi^\xi|\psi^\xi_\infty)$ (this is the matter of
Lemma~\ref{lem:fisher_macro_T} below), and thus [H4] is indeed satisfied
with $I_0=I(\psi^\xi(0,\cdot)|\psi^\xi_\infty)$ and $r=4
\pi^2$.

A fundamental lemma needed in the sequel is
\begin{lemma}\label{lem:A_t'-A'_x}
The difference between the biasing force $A_t'$ and the mean
force $A'$ can be expressed in term of the densities as
\begin{align*}
A_t'-A' &= \int_{\R} \partial_x \ln \left(\frac{\psi}{\psi_\infty} \right)
\frac{\psi}{\psi^\xi} \, dy - \partial_x \ln \left(\frac{\psi^\xi}{\psi^\xi_\infty} \right).
\end{align*}
\end{lemma}
\begin{proof}
This is a simple computation (using the fact that $\psi^\xi_\infty \equiv 1$):
\begin{align*}
\int_{\R} \partial_x \ln \left(\frac{\psi}{\psi_\infty} \right)
\frac{\psi}{\psi^\xi} \, dy - \partial_x \ln
\left(\frac{\psi^\xi}{\psi^\xi_\infty} \right)
&=\int_{\R} \partial_x \ln \psi \frac{\psi}{\psi^\xi} \, dy
- \int_{\R} \partial_x \ln \psi_\infty \frac{\psi}{\psi^\xi} \, dy -
\partial_x \ln \psi^\xi,\\
&= \int_{\R} \frac{\partial_x \psi}{\psi^\xi}\, dy + \int_{\R} \partial_x( V
- A) \frac{\psi}{\psi^\xi} \, dy -
\partial_x \ln \psi^\xi,\\
&= A_t' - A'.
\end{align*}
\end{proof}

We will also use the following two estimates:
\begin{lemma}\label{lem:estim_1_x}
Let us assume~[H2]--[H3]. Then, for
all $t \geq 0$, for
all $x \in \T$,
$$|A_t'(x)-A'(x)| \leq \|\partial_{x,y} V \|_{L^\infty} \sqrt{
  \frac{2}{\rho} e_m(t,x) }.$$
\end{lemma}
\begin{proof}
For any coupling measure $\pi \in \Pi(\mu_{t,x},\mu_{\infty,x})$, it holds:
\begin{align*}
|A_t'(x)-A'(x)|
&= \left|\int_{\R \times \R} \left( \partial_x V (x,y) - \partial_x V (x,y') \right) \pi(dy,dy')
\right|,\\
& \leq \|\partial_{x,y} V \|_{L^\infty} \int_{\R \times \R} |y-y'| \pi(dy,dy'),\\
& \leq \|\partial_{x,y} V \|_{L^\infty} \sqrt{\int_{\R \times \R} |y-y'|^2 \pi(dy,dy')}.
\end{align*}
Taking now the infimum over all $\pi \in \Pi(\mu_{t,x},\mu_{\infty,x})$
and using~[H3] together with Lemma~\ref{lem:otto_villani}, we obtain
\begin{equation*}
|A_t'(x)-A'(x)| \leq \|\partial_{x,y} V \|_{L^\infty}
W(\mu_{t,x},\mu_{\infty,x}) \leq \|\partial_{x,y} V \|_{L^\infty} \sqrt{ \frac{2}{\rho} H(\mu_{t,x} | \mu_{\infty,x}) },
\end{equation*}
which concludes the proof.
\end{proof}
\begin{lemma}\label{lem:estim_2_x}
Let us assume~[H3]. Then for all $t \geq 0$,
$$E_m(t) \leq \frac{1}{2 \rho} \int_{\T \times \R}  \left| \partial_y \ln \left( \frac{\psi}{\psi_\infty} \right)
\right|^2 \psi.$$
\end{lemma}
\begin{proof}
Using~[H3], it holds:
\begin{align*}
E_m
&= \int_{\T} e_m \psi^\xi \, dx, \\
&\leq  \int_{\T}  \frac{1}{2 \rho } \int_{\R} \left| \partial_y \ln \left(
  \frac{\psi}{\psi^\xi} \Big/
  \frac{\psi_\infty}{\psi^\xi_\infty} \right) \right|^2
\frac{\psi}{\psi^\xi} \, dy \, \psi^\xi \, dx,
\end{align*}
which yields the result since
$\psi^\xi / \psi^\xi_\infty$ does not depend on $y$.
\end{proof}

We are now in position to prove the exponential convergence of $E_m(t)$
to zero stated in Theorem~\ref{theo:CV}
(see Equation~(\ref{eq:CV_mic_ent})).

Equation~(\ref{eq:EDP_x}) on $\psi$ can be rewritten as:
$$\partial_t \psi=\div( \psi_\infty \nabla ( \psi / \psi_\infty)) +
  \partial_x( (A' - A_t') \psi).$$
Notice that the derivative $\frac{d E}{dt}$ can be obtained by
multiplying this equation by $\ln \left(
  \frac{\psi}{\psi_\infty} \right)$ and integrating over $\T \times \R$.
Thus, one obtains after some integrations by parts, using a
Cauchy-Schwarz inequality (to prove that~(\ref{eq:2}) is non positive) and Lemma~\ref{lem:A_t'-A'_x} (used twice):
\begin{align}
\frac{d E_m}{dt} &= \frac{d E}{dt} - \frac{d E_M}{dt},\nonumber\\
& = - \int_{\T} \int_{\R} \left| \nabla \ln \left( \frac{\psi}{\psi_\infty} \right)
\right|^2 \psi + \int_{\T} \int_{\R} (A_t' - A') \partial_x \ln \left(
  \frac{\psi}{\psi_\infty} \right) \psi + \int_{\T} \left| \partial_x \ln
  \left( \frac{\psi^\xi}{\psi^\xi_\infty} \right)
\right|^2 \psi^\xi,\nonumber\\
&= - \int_{\T} \int_{\R} \left| \partial_y \ln \left( \frac{\psi}{\psi_\infty} \right)
\right|^2 \psi \nonumber \\
&\quad - \int_{\T} \int_{\R} \left| \partial_x \ln \left( \frac{\psi}{\psi_\infty} \right)
\right|^2 \psi + \int_{\T} \left( \int_{\R} \partial_x \ln \left(
  \frac{\psi}{\psi_\infty} \right) \psi \, dy \right)^2
\frac{1}{\psi^\xi} \, dx \label{eq:2}\\
& \quad - \int_{\T} \int_{\R} \partial_x \ln \left(\frac{\psi^\xi}{\psi^\xi_\infty} \right)\partial_x \ln \left(
  \frac{\psi}{\psi_\infty} \right) \psi + \int_{\T} \left| \partial_x \ln
  \left( \frac{\psi^\xi}{\psi^\xi_\infty} \right)
\right|^2 \psi^\xi,\nonumber \\
& \leq - \int_{\T} \int_{\R} \left| \partial_y \ln \left( \frac{\psi}{\psi_\infty} \right)
\right|^2 \psi - \int_{\T}  \partial_x \ln
  \left( \frac{\psi^\xi}{\psi^\xi_\infty} \right)   \psi^\xi 
 (A_t'-A'). \nonumber
\end{align}
We now use Lemmas~\ref{lem:estim_1_x} and~\ref{lem:estim_2_x}:
\begin{align*}
\frac{d E_m}{dt}
& \leq - 2 \rho E_m +  \sqrt{\int_{\T} \left| A_t'-A' \right|^2
\psi^\xi} \sqrt{\int_{\T} \left| \partial_x \ln
  \left( \frac{\psi^\xi}{\psi^\xi_\infty} \right)
\right|^2 \psi^\xi},\\
 & \leq - 2 \rho E_m + \|\partial_{x,y} V \|_{L^\infty}  \sqrt{
  \frac{2}{\rho} E_m } \sqrt{I(\psi^\xi | \psi^\xi_\infty)}.
\end{align*}
Using [H4], we thus have:
$$
\frac{d  \sqrt{E_m}}{dt} \leq - \rho\sqrt{E_m} + \|\partial_{x,y} V
\|_{L^\infty} \sqrt{\frac{I_0}{2\rho}} \exp(-r t),
$$
from which we deduce~(\ref{eq:CV_mic_ent}).

Equation~(\ref{eq:CV_MF_1}) is then easily obtained using Lemma~\ref{lem:estim_1_x}.

\subsection{Proof of Proposition~\ref{prop:psi_xi} and Theorem~\ref{theo:CV} in the general case}\label{sec:proof_gen}

We now present the proof of Proposition~\ref{prop:psi_xi} and
Theorem~\ref{theo:CV} in the more general setting of
Section~\ref{sec:CV}. The proof follows the same lines as in the simple
case presented in Section~\ref{sec:proof_x}, but with additional
difficulties related to the geometry of the submanifolds~$\Sigma_z$.

We need the following result
\begin{lemma}\label{lem:psi_bar_prime}
The derivative of $\psi^\xi$ with respect to the reaction
coordinate value reads:
$$\partial_z \psi^\xi(t,z)=\int_{\Sigma_z} \left( 
  \frac{\nabla \xi \cdot \nabla \psi(t,\cdot)}{|\nabla \xi|^2} + \div
  \left(\frac{\nabla \xi}{|\nabla \xi|^2} \right) \psi(t,\cdot) 
\right)|\nabla \xi|^{-1} d\sigma_{\Sigma_z}.$$
\end{lemma}
\begin{proof}
For any smooth test function $g:\mathcal{M} \to \R$, we obtain (using the co-area
formula~(\ref{eq:co-area}) and an integration by parts):
\begin{align*}
\int_{\mathcal M} \psi^\xi(t,z)& g'(z) \, dz
=\int_{\mathcal D} \psi(t,x) g' \circ \xi (x) \, dx,\\
&=\int_{\mathcal D} \psi(t,x) \nabla (g \circ \xi) \cdot \nabla \xi |\nabla \xi|^{-2}
(x) \, dx,\\
&=- \int_{\mathcal D} \div\left(\frac{\psi(t,\cdot) \nabla \xi}{|\nabla \xi|^{2}} \right)
 g \circ \xi \, dx,\\
&=-\int_{\mathcal M} g(z) \int_{\Sigma_z} \left( 
  \frac{\nabla \xi \cdot \nabla \psi(t,\cdot)}{|\nabla \xi|^2} + \div
  \left(\frac{\nabla \xi}{|\nabla \xi|^2} \right) \psi(t,\cdot) 
\right)|\nabla \xi|^{-1} d\sigma_{\Sigma_z} \, dz,
\end{align*}
which yields the result.
\end{proof}
Using this lemma, it can be shown that $\psi^\xi$ satisfies a simple diffusion
equation, which is Proposition~\ref{prop:psi_xi}.
\begin{lemma}\label{lem:psi_bar}
The density $\psi^\xi$ satisfies the following diffusion equation on $\mathcal{M}$:
\begin{equation}\label{eq:psi_bar}
\partial_t \psi^\xi = \partial_{z}\left( W' \psi^\xi + \partial_{z} \psi^\xi \right).
\end{equation}
\end{lemma}
\begin{proof}
For any smooth test function $g:\mathcal{M} \to \R$, we have (using the co-area
formula~(\ref{eq:co-area}),~(\ref{eq:EDP}), an integration by parts and finally Lemma~\ref{lem:psi_bar_prime}):
\begin{align*}
\frac{d}{dt} &\int_{\mathcal M} \psi^\xi(t,\cdot) g \, dz
= \frac{d}{dt} \int_{\mathcal D} \psi(t,\cdot) g \circ \xi \, dx,\\
&= \int_{\mathcal D} \div\left(|\nabla \xi|^{-2} \left(\nabla (V-A_t\circ \xi+ W
    \circ \xi) \psi + 
  \nabla \psi \right)\right) g \circ \xi \, dx,\\
&= - \int_{\mathcal D}  |\nabla \xi|^{-2} \left(\nabla (V-A_t\circ \xi+ W
    \circ \xi) \psi + 
  \nabla \psi \right) \cdot \nabla \xi \, g' \circ \xi \, dx,\\
&= - \int_{\mathcal D}  |\nabla \xi|^{-2} \left(\nabla V \cdot \nabla \xi  \psi + 
  \nabla \psi \cdot \nabla \xi \right)  \, g' \circ \xi \, dx \\
&\quad + \int_{\mathcal D} A_t '\circ \xi g' \circ \xi \psi \, dx - \int_{\mathcal D} W' \circ \xi g' \circ \xi \psi \, dx,\\
&= - \int_{\mathcal M}  \int_{\Sigma_z} |\nabla \xi|^{-3} \left(\nabla V \cdot \nabla \xi  \psi + 
  \nabla \psi \cdot \nabla \xi \right)  \, d \sigma_{\Sigma_z} g'(z)  \,
dz \\
& \quad + \int_{\mathcal M} A_t '(z) g' (z)  \psi^\xi(z) \, dz - \int_{\mathcal M} W'(z) g' (z)  \psi^\xi(z) \, dz,\\
&=- \int_{\mathcal M}  \int_{\Sigma_z} \left( |\nabla \xi|^{-3}  
  \nabla \psi \cdot \nabla \xi   + \div(\nabla \xi |\nabla \xi|^{-2})
  |\nabla \xi|^{-1} \psi \right)\, d \sigma_{\Sigma_z} g'(z)  \,
dz\\
& \quad -  \int_{\mathcal M} W'(z) \psi^\xi(z) g' (z)   \, dz, \\
&= - \int_{\mathcal M} \left( \partial_z \psi^\xi(t,z) + W'(z) \psi^\xi(z)\right) g' (z)   \, dz,
\end{align*}
which is a weak formulation of~(\ref{eq:psi_bar}).
\end{proof}

As stated in Corollary~\ref{cor:W}, this result already yields the exponential
convergence to zero of the macroscopic Fisher information
$I(\psi^\xi|\psi^\xi_\infty)$ under adequate assumption on~$W$ (this is
the matter of [H4'] and Lemma~\ref{lem:fisher_macro_R} below). We
suppose in the following that [H4] is indeed satisfied.

The equivalent of Lemma~\ref{lem:A_t'-A'_x} writes
\begin{lemma}\label{lem:A_t'-A'}
The difference between the biasing force $A_t'$ and the mean
force $A'$ can be expressed in term of the densities as
\begin{align*}
A_t'(z)-A'(z) &= \int_{\Sigma_z} \frac{\nabla \xi}{|\nabla \xi|} \cdot \nabla \ln \left(\frac{\psi}{\psi_\infty} \right)
\frac{\psi}{\psi^\xi} |\nabla \xi|^{-2} \, d \sigma_{\Sigma_z} - \partial_z \ln \left(\frac{\psi^\xi}{\psi^\xi_\infty} \right).
\end{align*}
\end{lemma}
\begin{proof}
Using Lemma~\ref{lem:psi_bar_prime} and the definition of $A_t'$, it holds:
\begin{align*}
\int_{\Sigma_z} &\frac{\nabla \xi}{|\nabla \xi|} \cdot \nabla \ln \left(\frac{\psi}{\psi_\infty} \right)
\frac{\psi}{\psi^\xi} |\nabla \xi|^{-2} \, d \sigma_{\Sigma_z} -
\partial_z \ln \left(\frac{\psi^\xi}{\psi^\xi_\infty}
\right)\\
&=\int_{\Sigma_z} \frac{\nabla \xi}{|\nabla \xi|} \cdot \nabla \ln \psi
\frac{\psi}{\psi^\xi} |\nabla \xi|^{-2} \, d \sigma_{\Sigma_z}- \int_{\Sigma_z} \frac{\nabla \xi}{|\nabla \xi|} \cdot \nabla \ln \psi_\infty
\frac{\psi}{\psi^\xi} |\nabla \xi|^{-2} \, d \sigma_{\Sigma_z} \\
& \quad - \partial_z \ln \psi^\xi + \partial_z \ln \psi^\xi_\infty , \\
&=\frac{1}{\psi^\xi} \int_{\Sigma_z} \frac{\nabla \xi \cdot \nabla \psi}{|\nabla \xi|} 
|\nabla \xi|^{-2} \, d \sigma_{\Sigma_z} \\
&\quad + \int_{\Sigma_z} \frac{\nabla \xi}{|\nabla
  \xi|} \cdot \nabla \left( V - A \circ \xi + W \circ \xi\right)
\frac{\psi}{\psi^\xi} |\nabla \xi|^{-2} \, d \sigma_{\Sigma_z} -
\partial_z \ln \psi^\xi - W'(z), \\
&=\frac{\partial_z \psi^\xi}{\psi^\xi} -
\frac{1}{\psi^\xi} \int_{\Sigma_z} \div \left( \frac{\nabla \xi}{|\nabla \xi|^2}\right) 
|\nabla \xi|^{-1} \psi \, d \sigma_{\Sigma_z}+ \int_{\Sigma_z} \frac{\nabla \xi
  \cdot \nabla V}{|\nabla
  \xi|^3}
\frac{\psi}{\psi^\xi} \, d \sigma_{\Sigma_z} \\
& \quad - A'(z) -
\partial_z \ln \psi^\xi,\\
&=A_t'(z)-A'(z).
\end{align*}
\end{proof}
The equivalent of Lemmas~\ref{lem:estim_1_x} and~\ref{lem:estim_2_x} write:
\begin{lemma}\label{lem:estim_1}
Let us assume~[H2]--[H3]. Then for
all $t \geq 0$, for
all $z \in \mathcal{M}$,
$$|A_t'(z)-A'(z)| \leq M \sqrt{
  \frac{2}{\rho} e_m(t,z) }.$$
\end{lemma}
\begin{proof}
For any coupling measure $\pi \in \Pi(\mu_{t,z},\mu_{\infty,z})$ defined
on $\Sigma_z \times \Sigma_z$, it holds:
\begin{align*}
|A_t'(z)-A'(z)|
&= \left|\int_{\Sigma_z \times \Sigma_z} ( F(x) - F(x') ) \pi(dx,dx')
\right|,  \\
& \leq \left\|\nabla_{\Sigma_z} F  \right\|_{L^\infty} \sqrt{\int_{\Sigma_z \times \Sigma_z}
  d_{\Sigma_z}(x,x')^2 \pi(dx,dx')}. 
\end{align*}
%%%%%%%%%%%%%%%%%%%%%%%%%%%%%%%%%%%%%%%%%%%%%%%%%%%%%%%%%%%%%%%%%%%%
% Il suffit de considérer tous les chemins tracés sur $\Sigma_z$ qui
% mènent de $x$ à $x'$.
%%%%%%%%%%%%%%%%%%%%%%%%%%%%%%%%%%%%%%%%%%%%%%%%%%%%%%%%%%%%%%%%%%%%
Taking now the infimum over all $\pi \in \Pi(\mu_{t,z},\mu_{\infty,z})$
and using~[H2]--[H3] together with Lemma~\ref{lem:otto_villani},
we thus obtain
\begin{equation}\label{eq:1}
|A_t'(z)-A'(z)| \leq M W(\mu_{t,z},\mu_{\infty,z}) \leq M \sqrt{ \frac{2}{\rho} H(\mu_{t,z} | \mu_{\infty,z}) },
\end{equation}
which concludes the proof.
\end{proof}
\begin{lemma}\label{lem:estim_2}
Let us assume~[H3]. Then for all $t \geq 0$,
$$E_m(t) \leq \frac{1}{2 \rho} \int_{\mathcal D}  \left| \nabla_{\Sigma_z} \ln\left( \frac{\psi(t,\cdot)}{\psi_\infty}\right)
\right|^2 \psi.$$
\end{lemma}
\begin{proof}
Using~[H3], it follows:
\begin{align*}
E_m
&= \int_{\mathcal M} e_m \psi^\xi \, dz, \\
&\leq  \int_{\mathcal M}  \frac{1}{2 \rho } \int_{\Sigma_z}
\left|\nabla_{\Sigma_z} \ln\left( \frac{\psi(t,\cdot)}{\psi_\infty}\right)\right|^2 \frac{\psi(t,\cdot) |\nabla
\xi|^{-1} d \sigma_{\Sigma_z}}{\psi^\xi(t,z)} \, \psi^\xi \, dz,
\end{align*}
which yields the result, using the co-area formula~(\ref{eq:co-area}).
\end{proof}

We are now in position to prove the exponential convergence of $E_m(t)$ to
zero stated in
Theorem~\ref{theo:CV} (see Equation~(\ref{eq:CV_mic_ent})). Equation~(\ref{eq:EDP}) on $\psi$ can be
rewritten as:
$$\partial_t \psi=\div( |\nabla \xi|^{-2} \psi_\infty \nabla ( \psi / \psi_\infty)) +
  \div( |\nabla \xi|^{-2} \nabla \left( (A - A_t)\circ \xi\right) \, \psi).$$
Notice that the derivative $\frac{d E}{dt}$ can be obtained by
multiplying this equation by $\ln \left(
  \frac{\psi}{\psi_\infty} \right)$ and integrating over ${\mathcal D}$.
Thus, one obtains after some integrations by parts, using the co-area
formula~(\ref{eq:co-area}) and Lemma~\ref{lem:A_t'-A'}:
\begin{align*}
 \frac{d E_m}{dt} &=  \frac{d E}{dt} - \frac{d E_M}{dt},\\
& = - \int_{\mathcal D} \left| \nabla \ln \left( \frac{\psi}{\psi_\infty} \right)
\right|^2 |\nabla \xi|^{-2} \psi + \int_{\mathcal D} (A_t' - A')\circ \xi \, \nabla
\xi \cdot \nabla \ln \left(
  \frac{\psi}{\psi_\infty} \right) |\nabla \xi|^{-2} \psi \\
& \quad + \int_{\mathcal M} \left| \partial_z \ln
  \left( \frac{\psi^\xi}{\psi^\xi_\infty} \right)
\right|^2 \psi^\xi,\\
&= - \int_{\mathcal D} \left| \nabla_{\Sigma_z} \ln \left( \frac{\psi}{\psi_\infty} \right)
\right|^2 |\nabla \xi|^{-2} \psi - \int_{\mathcal D} \left( \frac{\nabla \xi}{|\nabla
    \xi|} \cdot \nabla \ln \left( \frac{\psi}{\psi_\infty} \right)
\right)^2 |\nabla \xi|^{-2} \psi \\
& \quad + \int_{\mathcal M} (A_t' - A')(z) \int_{\Sigma_z} \frac{\nabla
\xi}{|\nabla \xi|} \cdot \nabla \ln \left(
  \frac{\psi}{\psi_\infty} \right) |\nabla \xi|^{-2} \psi
d\sigma_{\Sigma_z} dz \\
& \quad + \int_{\mathcal M} \left| \partial_z \ln
  \left( \frac{\psi^\xi}{\psi^\xi_\infty} \right)
\right|^2 \psi^\xi, \\
&= - \int_{\mathcal D} \left| \nabla_{\Sigma_z} \ln \left( \frac{\psi}{\psi_\infty} \right)
\right|^2 |\nabla \xi|^{-2} \psi - \int_{\mathcal D} \left( \frac{\nabla \xi}{|\nabla
    \xi|} \cdot \nabla \ln \left( \frac{\psi}{\psi_\infty} \right)
\right)^2 |\nabla \xi|^{-2} \psi \\
& \quad + \int_{\mathcal M} \left( \int_{\Sigma_z} \frac{\nabla
\xi}{|\nabla \xi|} \cdot \nabla \ln \left(
  \frac{\psi}{\psi_\infty} \right) |\nabla \xi|^{-2} \psi
d\sigma_{\Sigma_z} \right)^2 (\psi^\xi)^{-1} dz \\
&\quad - \int_{\mathcal M}  \int_{\Sigma_z} \frac{\nabla
\xi}{|\nabla \xi|} \cdot \nabla \ln \left(
  \frac{\psi}{\psi_\infty} \right) |\nabla \xi|^{-2} \psi
d\sigma_{\Sigma_z}  \partial_z \ln
  \left( \frac{\psi^\xi}{\psi^\xi_\infty} \right) dz\\
&\quad + \int_{\mathcal M} \left| \partial_z \ln
  \left( \frac{\psi^\xi}{\psi^\xi_\infty} \right)
\right|^2 \psi^\xi.
\end{align*}
Using the Cauchy-Schwarz inequality:
\begin{align*}
 \Bigg( &\int_{\Sigma_z} \frac{\nabla
\xi}{|\nabla \xi|} \cdot \nabla \ln \left(
  \frac{\psi}{\psi_\infty} \right) |\nabla \xi|^{-1} \frac{ |\nabla \xi|^{-1} \psi
d\sigma_{\Sigma_z}}{\psi^\xi(z)} \Bigg)^2 \\
&\leq \int_{\Sigma_z} \left(  \frac{\nabla
\xi}{|\nabla \xi|} \cdot \nabla \ln \left(
  \frac{\psi}{\psi_\infty} \right) |\nabla \xi|^{-1} \right)^2 \frac{ |\nabla \xi|^{-1} \psi
d\sigma_{\Sigma_z}}{\psi^\xi(z)}
\end{align*}
 and Lemma~\ref{lem:A_t'-A'} again, we thus obtain
\begin{align*}
 \frac{d E_m}{dt} 
& \leq - \int_{\mathcal D} \left| \nabla_{\Sigma_z} \ln \left( \frac{\psi}{\psi_\infty} \right)
\right|^2 |\nabla \xi|^{-2} \psi - \int_{\mathcal M}  \partial_z \ln
  \left( \frac{\psi^\xi}{\psi^\xi_\infty} \right)   \psi^\xi 
 (A_t'-A').
\end{align*}
We now use~[H2], Lemmas~\ref{lem:estim_1} and~\ref{lem:estim_2}:
\begin{align*}
 \frac{d E_m}{dt} 
& \leq - 2 \rho m^{-2} E_m +  \sqrt{\int_{\mathcal M} \left| A_t'-A' \right|^2
\psi^\xi} \sqrt{\int_{\mathcal M} \left| \partial_z \ln
  \left( \frac{\psi^\xi}{\psi^\xi_\infty} \right)
\right|^2 \psi^\xi},\\
 & \leq - 2 \rho m^{-2} E_m + M \sqrt{
  \frac{2}{\rho} E_m }\sqrt{I(\psi^\xi | \psi^\xi_\infty)}.
\end{align*}
Using~[H4], we thus have:
$$
 \frac{d \sqrt{E_m}}{dt}  \leq - \rho m^{-2} \sqrt{E_m} + M
\sqrt{\frac{I_0}{2\rho}} \exp(- r t),
$$
from which we deduce~(\ref{eq:CV_mic_ent}).

Equation~(\ref{eq:CV_MF_1}) is then easily obtained using Lemma~\ref{lem:estim_1}.

\subsection{Proof of Corollary~\ref{cor:W}}\label{sec:proof_cor}

\subsubsection{Convergence of the macroscopic Fisher information}\label{sec:cv_info_fisher}

Let us first show that in both cases considered in~[H4'],
the exponential convergence~[H4] of the macroscopic
Fisher information indeed holds.

Let us first consider the case ${\mathcal M}=\T$ and $W=0$. We know
from~(\ref{eq:EDP_psi_xi}) that $\psi^\xi$ satisfies $\partial_t \psi^\xi =
\partial_{z,z} \psi^\xi$ on $\T$, and we would like to show exponential
convergence of the Fisher information $I(\psi^\xi(t,\cdot)|\psi^\xi_\infty)$.

\begin{lemma}[Convergence of the Fisher information when ${\mathcal
    M}=\T$ and $W=0$]\label{lem:fisher_macro_T}
Let $\phi$ be a function defined for $t \geq 0$ and $x \in \T$ which satisfies
$$\partial_t \phi = \partial_{x,x} \phi \text{ on $\T$}$$
and such that $\int_{\T} \phi(0,\cdot)=1$, $\phi(0,\cdot)$ is non
negative, and $I(\phi(0,\cdot)|\phi_\infty)<\infty$, where $\phi_\infty
\equiv 1$ is the longtime limit of $\phi$. Then, $\forall t \geq 0$,
$$I(\phi(t,\cdot)|\phi_\infty)\leq
I(\phi(0,\cdot)|\phi_\infty) \exp(-8 \pi^2 t).$$
\end{lemma}
\begin{proof}
Let us denote $u=\sqrt{\phi}$.
We notice that $I(\phi|\phi_\infty)=\int_{\T} |\partial_x \ln \phi|^2 \phi =4 \int_{\T}
|\partial_x u|^2$. Moreover, we have from~(\ref{eq:psi_bar_x})
$$\partial_t u = \partial_{x,x} u + \frac{(\partial_x u)^2}{u}.$$
Therefore,
\begin{align*}
\frac{d}{dt} \int_{\T} (\partial_x u)^2
&= 2 \int_{\T} \partial_{x,x,x} u \, \partial_x u + 2 \int_{\T} \partial_x\left(
  \frac{(\partial_x u)^2}{u}\right) \partial_x u,\\
&= - 2 \int_{\T} (\partial_{x,x} u)^2 -2 \int_{\T}  \frac{(\partial_x
  u)^2}{u}\partial_{x,x} u,\\
&= - 2 \int_{\T} (\partial_{x,x} u)^2 -2 \int_{\T}  \frac{\partial_x ( (\partial_x
  u)^3 )}{3u},\\
&= - 2 \int_{\T} (\partial_{x,x} u)^2 - \frac{2}{3} \int_{\T}  \frac{(\partial_x
  u)^4}{u^2},\\
& \leq - 8 \pi^2 \int_{\T} (\partial_x u)^2,
\end{align*}
where we have used the Poincaré-Wirtinger inequality on $\T$, applied to
$\partial_x u$: For any function $f
\in H^1(\T)$,
$$\int_{\T} \left(f - \int_{\T} f\right)^2 \leq \frac{1}{4\pi^2} \int_{\T} (\partial_x
f)^2.$$
\end{proof}

Let us now consider the case ${\mathcal M}=\R$ and $W \neq 0$ which is such that $W''$ is
  bounded from below and $\frac{\exp(-\beta
    W)}{\int_{\mathcal M} \exp(-\beta W)}$ satisfies a logarithmic Sobolev
  inequality (as stated in~[H4']). We know
from~(\ref{eq:EDP_psi_xi}) that $\psi^\xi$ satisfies $\partial_t \psi^\xi =
\partial_{z}\left( W' \psi^\xi + \partial_{z} \psi^\xi \right)$ on~$\R$, and we would like to show exponential
convergence of the Fisher information~$I(\psi^\xi(t,\cdot)|\psi^\xi_\infty)$.

\begin{lemma}[Convergence of the Fisher information when ${\mathcal
    M}=\R$ and $W\neq0$]\label{lem:fisher_macro_R}
Let $\phi$ be a function defined for $t \geq
0$ and $x \in \R$ which satisfies
$$\partial_t \phi = \partial_{x} \left(W' \phi + \partial_x \phi \right) \text{ on $\R$,}$$
and such that $\int_{\R} \phi(0,\cdot)=1$, $\phi(0,\cdot)$ is non
negative, and $I(\phi(0,\cdot)|\phi_\infty)<\infty$, where $\phi_\infty
\equiv \frac{\exp(-W)}{\int_\R \exp(-W)}$ is the
longtime limit of $\phi$. Let us assume that $W''$ is
  bounded from below by a constant $\alpha$ and $\phi_\infty$ satisfies LSI($\overline{r}$), with $\overline{r}>0$. We can
  suppose without loss of generality that
$$\overline{r} \geq \alpha.$$
Then there exists $I_0>0$ and $r>0$ such that $\forall t \geq 0$,
$$I(\phi(t,\cdot)|\phi_\infty)\leq I_0 \exp(- 2 r t).$$
More precisely, when $\alpha = \overline{r} > 0$, it is possible to take
$I_0=I(\phi(0,\cdot)|\phi_\infty)$ and $r=\alpha$. When  $\alpha <
\overline{r}$, for any $ \varepsilon \in (0, \overline{r})$, it is possible
to choose $r=\overline{r} -
\varepsilon$ for a well-chosen constant $I_0>0$.
\end{lemma}
\begin{proof}
The fact that $\overline{r} \geq \alpha$ is clear since either $\alpha
\leq 0$, or $\alpha >0$ in which case it is well-known that $\phi_\infty$
satisfies LSI($\alpha$) (see for example~\cite{ABC-00}), so that one can choose at least $\overline{r}=\alpha$.

Let us recall the expression for the entropy $H(\phi(t,\cdot)|\phi_\infty)= \int_\R
\ln(\phi/\phi_\infty) \phi$ and the Fisher information
$I(\phi(t,\cdot)|\phi_\infty)=\int_\R |\partial_x \ln(\phi/\phi_\infty)|^2
\phi$. Since $\phi_\infty$ satisfies LSI($\overline{r}$), we have
$$H(\phi(t,\cdot)|\phi_\infty) \leq \frac{1}{2 \overline{r}}
I(\phi(t,\cdot)|\phi_\infty).$$
Moreover, by standard computations (see for example~\cite{arnold-markowich-toscani-unterreiter-01}), we have
$$\frac{d}{dt}H(\phi(t,\cdot)|\phi_\infty)=-I(\phi(t,\cdot)|\phi_\infty)$$
and
\begin{equation}\label{eq:dI_dt}
\frac{d}{dt}I(\phi(t,\cdot)|\phi_\infty)=-2 \int_\R
\frac{\phi}{\phi_\infty} \left|\partial_{x,x} \ln \left(
    \frac{\phi}{\phi_\infty}\right) \right|^2 \phi_\infty
-2 \int_\R\frac{\phi}{\phi_\infty} \left|\partial_{x} \ln\left(
\frac{\phi}{\phi_\infty}\right)\right|^2 W'' \phi_\infty.
\end{equation}
If $\alpha = \overline{r}$, we thus obtain from~(\ref{eq:dI_dt}) that
$\frac{d}{dt}I(\phi(t,\cdot)|\phi_\infty) \leq - 2 \alpha
I(\phi(t,\cdot)|\phi_\infty)$ which concludes the proof in this case.

Let us now suppose that $\alpha < \overline{r}$. The technique of proof
we propose is taken from~\cite{villani-06}. For any $\lambda >0$, we have
\begin{align*}
\frac{d}{dt} &\left( H(\phi(t,\cdot)|\phi_\infty) + \lambda
  I(\phi(t,\cdot)|\phi_\infty) \right)\\
&= - \int_\R\frac{\phi}{\phi_\infty} \left|\partial_{x} \ln\left(
\frac{\phi}{\phi_\infty}\right)\right|^2 \phi_\infty -2 \lambda \int_\R
\frac{\phi}{\phi_\infty} \left|\partial_{x,x} \ln \left(
    \frac{\phi}{\phi_\infty}\right) \right|^2 \phi_\infty\\
& \quad -2 \lambda \int_\R\frac{\phi}{\phi_\infty} \left|\partial_{x} \ln\left(
\frac{\phi}{\phi_\infty}\right)\right|^2 W'' \phi_\infty,\\
& \leq - \int_\R (1 + 2 \lambda W'') \frac{\phi}{\phi_\infty} \left|\partial_{x} \ln\left(
\frac{\phi}{\phi_\infty}\right)\right|^2 \phi_\infty, \\
& \leq - (1 + 2 \lambda \inf W'') I(\phi(t,\cdot)|\phi_\infty),\\
& \leq - \frac{1 + 2  \alpha \lambda}{\lambda + 1/(2 \overline{r})} \left( H(\phi(t,\cdot)|\phi_\infty) + \lambda
  I(\phi(t,\cdot)|\phi_\infty) \right).
\end{align*}
We thus obtain that, for any $\lambda >0$,
$$H(\phi(t,\cdot)|\phi_\infty) + \lambda I(\phi(t,\cdot)|\phi_\infty)
\leq \big( H(\phi(0,\cdot)|\phi_\infty) + \lambda
I(\phi(0,\cdot)|\phi_\infty) \big) \,
\exp\left( - \frac{1 + 2 \alpha \lambda}{\lambda + 1/(2 \overline{r})} t\right),
$$
and therefore
$$I(\phi(t,\cdot)|\phi_\infty)\leq
\left(\frac{1}{\lambda}H(\phi(0,\cdot)|\phi_\infty) +
  I(\phi(0,\cdot)|\phi_\infty) \right)\exp\left( - \frac{1 + 2 \alpha \lambda}{\lambda + 1/(2 \overline{r})} t\right).
$$
Since $\frac{1 + 2 \alpha \lambda}{\lambda + 1/(2 \overline{r})}$
goes to $2 \overline{r}$ when $\lambda$ goes to $0$, for any $\varepsilon
\in (0, \overline{r})$, one can find a $\lambda>0$ such that $\frac{1 + 2
  \alpha \lambda}{\lambda + 1/(2 \overline{r})}=2(\overline{r} -
\varepsilon)$, which concludes the proof.
\end{proof}

\subsubsection{Convergence of the biasing force}\label{sec:cv_A}

Let us now prove the convergence result~(\ref{eq:CV_MF_2}) for the biasing force.

In the case ${\mathcal M}=\T$ (and thus $W=0$), we can prove
the convergence of  $\|A_t' - A'\|_{L^2(\T)}$ to zero in the following sense (which
implies~(\ref{eq:CV_MF_2}), using~(\ref{eq:CV_mic_ent})): for any $\varepsilon \in (0,1)$, $\forall t \geq t_\varepsilon$,
\begin{equation}\label{eq:CV_MF_2_T}
 \|A_t' - A'\|_{L^2(\T)}^2 \leq \frac{2}{1-\varepsilon}\frac{M^2}{\rho}
 E_m(t),
\end{equation}
where $\displaystyle{t_\varepsilon=\min\left(0, (4 \pi^2)^{-1} \ln\left( \varepsilon^{-1}
    \sqrt{\int_\T(\partial_z \psi^\xi(0,\cdot))^2} \right)\right)}$. This is obtained using the fact that $\int_\T (\partial_x
\psi^\xi(t,\cdot))^2 \leq \int_\T (\partial_x
\psi^\xi(0,\cdot))^2 \exp(-8 \pi^2 t)$ (the proof of this estimate is similar to the
one of Lemma~\ref{lem:fisher_macro_T}) and the fact that for any function
$f \in H^1(\T)$,
$$\left\| f - \int_\T f \right\|_{L^\infty}^2 \leq \int_\T (\partial_x f)^2,$$
applied to $f=\psi^\xi$. Thus we have $\left\| \psi^\xi - 1 \right\|_{L^\infty}^2 \leq \int_\T (\partial_x
\psi^\xi(0,\cdot))^2 \exp(-8 \pi^2 t)$ which implies that for $t \geq
t_\varepsilon$, $\psi^\xi(t,\cdot) \geq 1- \varepsilon$ which
yields~(\ref{eq:CV_MF_2_T}) from~(\ref{eq:CV_MF_1}).

%%%%%%%%%%%%%%%%%%%%%%%%%%%%%%%%%%%%%%%%%%%%%%%%%%%%%%%%
% Soit x_0 tel que f(x_0)= \int_\T f. On a alors
% f(x) - \int_\T f = \int_{x_0}^x f'(u) \, du
%               \leq \| f' \|_{L^2(\T)}
%%%%%%%%%%%%%%%%%%%%%%%%%%%%%%%%%%%%%%%%%%%%%%%%%%%%%%%%

Let us now prove~(\ref{eq:CV_MF_2}) in the case ${\mathcal M}=\R$, under
assumption~[H4'] on $W$.
Let us introduce a compact $K \subset \mathcal{M}$. Since $L^\infty(K)
\subset H^1(K)$ (with continuous injection), there exists $c>0$ such that
\begin{align*}
\left\|\frac{\psi^\xi}{\psi^\xi_\infty} -1 \right\|_{L^\infty(K)} 
&\leq c \left(\left\|\frac{\psi^\xi}{\psi^\xi_\infty} -1
  \right\|_{L^2(K)} +  \left\|\partial_z \left(\frac{\psi^\xi}{\psi^\xi_\infty} -1
 \right) \right\|_{L^2(K)}  \right),\\
& \leq \frac{c}{\inf_K \sqrt{\psi^\xi_\infty}} \left( \sqrt{\int_\R \left(\frac{\psi^\xi}{\psi^\xi_\infty} -1
  \right)^2 \psi^\xi_\infty}  + \sqrt{\int_\R \left(\partial_z \left(\frac{\psi^\xi}{\psi^\xi_\infty} -1
  \right) \right)^2 \psi^\xi_\infty} \right).
\end{align*}
Thus, for any $ \varepsilon \in (0,\overline{r})$, there exists $C>0$ such that
\begin{equation*}
\left\|\frac{\psi^\xi}{\psi^\xi_\infty} -1 \right\|_{L^\infty(K)} \leq C
\exp(- r t),
\end{equation*}
with $r=\overline{r}-\varepsilon$. This inequality is obtained from the fact that since
$\psi^\xi_\infty$ satisfies LSI($\overline{r}$), then $\psi^\xi_\infty$
also satisfies a Poincar\'e inequality with the same constant
$\overline{r}$ (see for example~\cite{ABC-00}), and a proof similar to
that of Lemma~\ref{lem:fisher_macro_R} for the convergence of the Fisher information $\displaystyle{\int_\R \left(\partial_z \left(\frac{\psi^\xi}{\psi^\xi_\infty} -1 \right) \right)^2 \psi^\xi_\infty}$
associated with the Poincar\'e inequality. Now, we write
\begin{align*}
\int_K |A_t' - A'| \psi^\xi_\infty
&= \int_K |A_t' - A'| \psi^\xi - \int_K |A_t' - A'| \left(
  \frac{\psi^\xi}{\psi^\xi_\infty} -1 \right) \psi^\xi_\infty,\\
& \leq \int_\R |A_t' - A'|^2 \psi^\xi + C \exp(- r t)\int_K |A_t' - A'| \psi^\xi_\infty.
\end{align*}
Thus, for $t$ sufficiently large, $\int_K |A_t' - A'| \psi^\xi_\infty$
is bounded from above by some constant times $\int_\R |A_t' - A'|^2 \psi^\xi$, which
yields~(\ref{eq:CV_MF_2}) (using~(\ref{eq:CV_MF_1}) and~(\ref{eq:CV_mic_ent})).

%%%%%%%%%%%%%%%%%%%%%%%%%%%%%%%%%%%%%%%%%%%%%%%%%%%%%%%%%%%%%%%%%%%%%%%%%%%%%%%
% CV de A_t' sous l'hypothèse F borné
%%%%%%%%%%%%%%%%%%%%%%%%%%%%%%%%%%%%%%%%%%%%%%%%%%%%%%%%%%%%%%%%%%%%%%%%%%%%%%%
% Une autre preuve possible sous l'hypothèse $\exists C>0$,
% $$\left\|\left( \frac{\nabla V \cdot \nabla \xi}{|\nabla \xi|^{2}}
%   - \beta^{-1} \div\left(\frac{\nabla \xi}{|\nabla \xi|^{2}}\right)
% \right)\right\|_{L^\infty( \bigcup_{z \in K} \Sigma_z} \leq C.$$
% Dans ce cas, on a $\|A_t'\|_{L^\infty( \bigcup_{z \in K} \Sigma_z} \leq
% C$ et $\|A'\|_{L^\infty( \bigcup_{z \in K} \Sigma_z} \leq
% C$. On en déduit:
% \begin{align*}
% \int|A_t' - A'| \psi^\xi_\infty
% &= \int|A_t' - A'| \psi^\xi - \int|A_t' - A'| \left( \frac{\psi^\xi}{\psi^\xi_\infty} - 1\right)\psi^\xi_\infty,\\
% &\leq \sqrt{\int|A_t' - A'|^2 \psi^\xi} + \frac{1}{4C} \int|A_t' - A'|^2
% \psi^\xi_\infty + C \int\left( \frac{\psi^\xi}{\psi^\xi_\infty} -
%   1\right)^2\psi^\xi_\infty\\
% &\leq \sqrt{\int|A_t' - A'|^2 \psi^\xi} + \frac{2C}{4C} \int|A_t' - A'|
% \psi^\xi_\infty + C \int\left( \frac{\psi^\xi}{\psi^\xi_\infty} -
%   1\right)^2\psi^\xi_\infty.
% \end{align*}
% We thus have
% $$
% \int|A_t' - A'| \psi^\xi_\infty \leq 2 \sqrt{\int|A_t' - A'|^2 \psi^\xi} + 2C \int\left( \frac{\psi^\xi}{\psi^\xi_\infty} -
%   1\right)^2\psi^\xi_\infty.
% $$
% Et ces deux termes convergent bien exponentiellement vite vers $0$.
%%%%%%%%%%%%%%%%%%%%%%%%%%%%%%%%%%%%%%%%%%%%%%%%%%%%%%%%%%%%%%%%%%%%%%%%%%%%%%%

\subsection{Proof of Theorem~\ref{theo:CV_ori}}\label{sec:proof_original}

Let us now prove Theorem~\ref{theo:CV_ori}. We still assume, up to a
change of variable, that $\beta=1$. We have:
\begin{align*}
\frac{d E}{dt}
&=-\int_{\mathcal D} \left|\nabla \ln\left(\frac{\psi}{\psi_\infty}\right)\right|^2
\psi + \int_{\mathcal D} ( A_t' - A')\circ \xi \nabla \xi \cdot \nabla
\ln\left(\frac{\psi}{\psi_\infty}\right) \psi,\\
& \leq -\int_{\mathcal D} \left|\nabla \ln\left(\frac{\psi}{\psi_\infty}\right)\right|^2
\psi + \sqrt{\int_{\mathcal M} | A_t' - A' |^2 \psi^\xi} \sqrt{\int_{\mathcal D}
    \left| \nabla \xi \cdot \nabla
    \ln\left(\frac{\psi}{\psi_\infty}\right) \right|^2  \psi}.
\end{align*}
Since, by Lemmas~\ref{lem:estim_1} and~\ref{lem:estim_2},
\begin{align*}
\int_{\mathcal M} | A_t' - A' |^2 \psi^\xi
&\leq \frac{M^2}{\rho} \int_{\mathcal D} \left|\nabla_{\Sigma_z} \ln
  \left(\frac{\psi}{\psi_\infty} \right) \right|^2 \psi,
\end{align*}
we thus obtain
\begin{align*}
\frac{d E}{dt}
& \leq -\int_{\mathcal D} \left|\nabla \ln\left(\frac{\psi}{\psi_\infty}\right)\right|^2
\psi + \frac{M m}{\sqrt{\rho}} \sqrt{\int_{\mathcal D} \left|\nabla_{\Sigma_z} \ln
  \left(\frac{\psi}{\psi_\infty} \right) \right|^2 \psi} \sqrt{\int_{\mathcal D}
    \left| \frac{\nabla \xi}{|\nabla \xi|} \cdot \nabla
    \ln\left(\frac{\psi}{\psi_\infty}\right) \right|^2  \psi}\\
& \leq \left(-1 +  \frac{M m}{2\sqrt{\rho}} \right) \int_{\mathcal D} \left|\nabla
  \ln\left(\frac{\psi}{\psi_\infty}\right)\right|^2 \psi,
\end{align*}
where we have used the fact that, for any function $f:{\mathcal D} \to \R$, $|\nabla f|^2 = | \nabla_{\Sigma_z}
f|^2 + \left|\frac{\nabla \xi}{|\nabla \xi|} \cdot \nabla f \right|^2$.
The logarithmic Sobolev inequality with respect to $\psi_\infty$
(see~[H5]) concludes the proof.

\appendix
\section{The co-area formula}\label{sec:co-area}

The aim  of this section  is to state  the co-area formula for  a
function  $\xi :  {\mathcal D}  \to \R^p$,  (where $1  \leq  p <  n$) such  that
$\text{rank}(\nabla \xi)=p$. Classical proofs for the co-area formula can
be                    found                    in                    the
books~\cite{ambrosio-fusco-pallara-00,evans-gariepy-92}.   These  proofs
are however quite involved  since they assume only Lipschitz-regularity
for $\xi$.  The proof is simpler in the case of a
smooth~$\xi$: it can be done by an adequate parameterization and a simple change of variables. 

\begin{lemma}[co-area formula]
For any smooth function $\phi: \R^{n}  \to \R$,
\begin{equation}
\label{eq:co-area_p}
\int_{\R^{n}} \phi(x) \sqrt{\det G(x)} dx = \int_{\R^p}
\int_{\Sigma_z} \phi \, d\sigma_{\Sigma_z} \, dz,
\end{equation}
where $G$ is a $p\times p$ matrix with $G_{i,j}=\nabla \xi_i \cdot
\nabla \xi_j$. In the case $p=1$, Equation~(\ref{eq:co-area_p}) reads:
\begin{equation}
\label{eq:co-area}
\int_{\R^{n}} \phi(x) |\nabla \xi|(x) dx = \int_{\R}
\int_{\Sigma_z} \phi \, d\sigma_{\Sigma_z} \, dz,
\end{equation}
\end{lemma}
\begin{remark}
This formula shows that if the random variable $X$ has law
$\psi(x)\, dx$ in~$\R^n$, then $\xi(X)$ has law
$$\int_{\Sigma_z} \psi \, (\det G)^{-1/2} \,d\sigma_{\Sigma_z} \,dz,$$ and
the law of $X$ conditioned to a fixed value $z$ of $\xi(X)$ is
$$d\mu_{z}=\frac{\psi \, (\det G)^{-1/2}
  \,d\sigma_{\Sigma_z}}{\int_{\Sigma_z} \psi \, (\det G)^{-1/2}
  \,d\sigma_{\Sigma_z}}.$$ Indeed, for any bounded functions $f$
and $g$,
\begin{align*}
\E(f(\xi(X))g(X))&=\int_{\R^n} f(\xi(x)) g(x) \psi(x)\, dx,\\
&=\int_{\R^p} \int_{\Sigma_z} f \circ \xi \, g\, \psi\, (\det G)^{-1/2} d
\sigma_{\Sigma_z} \, dz,\\
&=\int_{\R^p} f(z) \frac{\int_{\Sigma_z} g \, \psi\,(\det G)^{-1/2} d
\sigma_{\Sigma_z}}{\int_{\Sigma_z} \psi\,(\det G)^{-1/2} d
\sigma_{\Sigma_z}} \int_{\Sigma_z} \psi\,(\det G)^{-1/2} d
\sigma_{\Sigma_z} \, dz.
\end{align*}
The measure $(\det G)^{-1/2} d
\sigma_{\Sigma_z}$ is sometimes denoted by $\delta_{\xi(x)-z}$ in the literature.
\end{remark}

\section{Another possible set of assumptions for the convergence of the
  adaptive dynamics~(\ref{eq:X})--(\ref{eq:A_prime_t})}\label{sec:rieman}

It is also possible to state a result similar to Theorem~\ref{theo:CV}
for the dynamics~(\ref{eq:X})--(\ref{eq:A_prime_t}) under slightly different
assumptions than~[H2] and~[H3] by introducing
another Riemannian structure on $\Sigma_z$ (see~\cite{otto-villani-00})
than that induced by the scalar product of the ambient space~${\mathcal D}$. Let us introduce the following scalar product: $\forall x \in \Sigma_z$, $\forall u , v \in T_x \Sigma_{z}$,
\begin{equation}\label{eq:riem}
\langle u , v \rangle_{\Sigma_z}= u \cdot v |\nabla \xi|^{2}(x),
\end{equation}
where $\cdot$ denotes as before the scalar product of the ambient
space~${\mathcal D}$, and the associated norm: $\forall x \in \Sigma_z$, $\forall u \in T_x \Sigma_{z}$,
$$|u|_{\Sigma_z}^2=\langle u , u \rangle_{\Sigma_z}= |u|^2 |\nabla
\xi|^{2}(x).$$
Accordingly, the definition of the surface gradient is modified as
follows\footnote{With a slight abuse of notation, we still use the same
  notation $\nabla_{\Sigma_z}$ to denote the surface gradient, or
  $I(\mu_{t,z} | \mu_{\infty,z})$ to denote the Fisher information, or
  $d_{\Sigma_z}$ to denote the geodesic distance, or $\rho$ to denote
  the microscopic rate of convergence, while these are not
  the same as in the rest of the paper, since the Riemannian structure
  has been changed.} (compare with~(\ref{eq:grad_surf})): For $f: {\mathcal D} \to \R$,
\begin{equation}\label{eq:grad_surf_2}
\nabla_{\Sigma_z} f = |\nabla \xi|^{-2} P \nabla f.
\end{equation}
% En effet, pour tout $u \in T_x\Sigma$, $Df_x(u)=P\nabla f \cdot u$. Par
% ailleurs, par définition de $\nabla_{\Sigma_z}$, $Df_x(u)=\langle
% \nabla_{\Sigma_z} f, u \rangle_{\Sigma_z}=|\nabla \xi|^{2}(x) \nabla_{\Sigma_z} f \cdot u$. D'où le résultat.
In particular, we have $|\nabla_{\Sigma_z} f|_{\Sigma_z} =|\nabla \xi|^{-1} |P \nabla f|$.

In this case, the Fisher information between the conditioned measures
$\mu_{t,z}$ and $\mu_{\infty,z}$ is (see~\cite{otto-villani-00}):
\begin{align*}
I(\mu_{t,z} | \mu_{\infty,z})
&=\int_{\Sigma_z}
\left|\nabla_{\Sigma_z} \ln\left( \frac{\psi(t,\cdot)}{\psi_\infty}
\right)\right|_{\Sigma_z}^2  \frac{\psi(t,\cdot) |\nabla
\xi|^{-1} d \sigma_{\Sigma_z}}{\psi^\xi(t,z)},\\
&=\int_{\Sigma_z}
\left |P \nabla \ln\left( \frac{\psi(t,\cdot)}{\psi_\infty}
\right)\right|^2   |\nabla \xi|^{-2} \frac{\psi(t,\cdot) |\nabla
\xi|^{-1} d \sigma_{\Sigma_z}}{\psi^\xi(t,z)},
\end{align*}
and the assumption~[H3] is stated in terms of this new
Fisher information:
\begin{equation*}\label{eq:hyp_LSI_prime}
\text{{\bf [H3']}~~~}
\left\{
\begin{array}{c}
\text{$V$ and $\xi$ are such that $\exists \rho >0$, for all $z \in
  {\mathcal M}$,}\\
\text{ the conditional measure $\mu_{\infty,z}$ satisfies LSI($\rho$),}\\
\text{ $\Sigma_z$ being endowed with the Riemannian structure~(\ref{eq:riem}).}
\end{array}
\right.
\end{equation*}

Using this Fisher information, Lemma~\ref{lem:estim_2} writes:
$$E_m(t) \leq \frac{1}{2 \rho} \int_{\mathcal D}  \left| P \nabla \ln\left( \frac{\psi(t,\cdot)}{\psi_\infty}\right)
\right|^2 |\nabla \xi|^{-2} \psi.$$
The definition for the Wasserstein distance is now stated using the
geodesic distance~$d_{\Sigma_z}$: $\forall x,y \in \Sigma_z$,
$$d_{\Sigma_z}(x,y)=\inf \left\{ \sqrt{\int_0^1
    |\dot{w}(t)|_{\Sigma_z}^2 \, dt} \,
  \Bigg| \,  w \in
  {\mathcal C}^1([0,1],\Sigma_z),\, w(0)=x,\, w(1)=y \right\}.$$
Thus, the estimate of Lemma~\ref{lem:estim_1} is changed to:
\begin{align*}
|A_t'(z)-A'(z)|
&= \left|\int_{\Sigma_z \times \Sigma_z} (F(x) - F(x')) \pi(dx,dx')
\right|,\\
& \leq \left\|  |\nabla \xi|^{-1} \left|P \nabla F\right|  \right\|_{L^\infty} \sqrt{\int_{\Sigma_z \times \Sigma_z} d_{\Sigma_z}(x,x')^2 \pi(dx,dx')},
\end{align*}
where $F$ is defined by~(\ref{eq:F}). Notice that
$$|\nabla \xi|^{-1} \left|P \nabla F \right|= \left| \nabla_{\Sigma_z} F\right|_{\Sigma_z}.
$$
Thus, assumption~[H2] is modified as:
\begin{equation*}\label{eq:hyp_V_prime}
\text{{\bf [H2']}~~~}
\left\{
\begin{array}{c}
\text{$V$ and $\xi$ are sufficiently differentiable functions such that} \\
\text{$\left\| \left| \nabla_{\Sigma_z} F \right|_{\Sigma_z}  \right\|_{L^\infty} \leq M<\infty$.}
\end{array}
\right.
\end{equation*}
The rest of the proof remains the same, and exponential convergence is
thus obtained, assumptions~[H2] and~[H3]
being respectively replaced by~[H2']
and~[H3']. With this set of assumptions, the rate of convergence is
$\lambda=\beta^{-1} \min( \rho,r)$.

% Noter que~[H2'] a l'avantage d'être intrinsèque.
% Ensuite, le papier~\cite{otto-villani-00} montre que LSI($\rho$)
% implique T($\rho$) avec n'importe quelle mesure riemannienne.

% En utilisant ces estimées on a donc:
% \begin{align*}
% \frac{d E_m}{dt}
% & \leq - \int \left| P \nabla \ln \left( \frac{\psi}{\psi_\infty} \right)
% \right|^2 |\nabla \xi|^{-2} \psi - \int  \partial_z \ln
%   \left( \frac{\psi^\xi}{\psi^\xi_\infty} \right)   \psi^\xi 
%  (A_t'-A'),\\
% & \leq - 2 \rho E_m + M \sqrt{I(\psi^\xi | \psi^\xi_\infty)}\sqrt{
%   \frac{2}{\rho} E_m }.
% \end{align*}
% Cela évite d'avoir la borne sur $|\nabla \xi|$ qui apparaisse
% partout... d'autant plus que $|\nabla \xi| \not\in L^\infty$ dans le cas
% de l'angle...

% Par contre, cela rend l'hypothèse LSI($\rho$) plus abstraite, et le
% gradient surfacique n'est plus la projection du gradient...

% Réfléchir aux modifications quand $\xi$ est multiplié par une constante
% pour voir si c'est naturelle que une borne sur $|\nabla \xi|$ apparaisse
% dans le taux de convergence (comme c'est le cas actuellement).

% Pour que tout cela marche bien (existence de la distance géodésique, par
% exemple), il faut peut-être une hypothèse du style $|\nabla \xi|$ minoré
% ou majoré ??? (Penser au cas de l'angle ou $|\nabla \xi|$ explose).

% (La métrique riemanienne ($\langle u , v \rangle_{\Sigma_z}$) et la structure différentielle
% ($T_x\Sigma_z$) sur la sous-variété $\Sigma_z$ sont deux
% structures qui n'interagissent pas).

% Par contre, bien sûr, la définition du gradient surfacique dépend du
% choix du produit scalaire.

\acknowledgement{This work is supported by the ANR INGEMOL of the
  French Ministry of Research. TL would like to thank Ch.~Chipot who initiated this
  work by a question about the ABF method. Part of this work was completed
  during a summer school of the GDR CHANT. We would like to
  thank F.~Castella for the organization of this school. We would like to thank C. Villani for pointing
  out~\cite{villani-06} to prove Lemma~\ref{lem:fisher_macro_R}.}

% \bibliography{biblio_HD,biblio_MD,ma_biblio}
% \bibliographystyle{plain}

\end{document}